\documentclass[11pt]{article}

\usepackage{geometry}
\usepackage{graphicx}
\usepackage{amsmath, amsthm, mathrsfs}
\usepackage{hyperref}
\usepackage{amssymb}                            
\usepackage[dvips]{epsfig}
\usepackage{subfig}
\usepackage[noadjust]{cite}
\usepackage{tikz}
\usepackage{caption}
\usepackage{epstopdf}
\usepackage{color}
\usepackage{multirow}

\newcommand{\lls}{\langle\langle}
\newcommand{\ggs}{\rangle\rangle}

\newtheorem{theorem}{\textbf{Theorem}}
\newtheorem*{theorem*}{\textbf{Theorem}}

\newtheorem{prop}{\textbf{Proposition}}
\newtheorem{assumption}{\textbf{Assumption}}
\newtheorem{cor}{Corollary}
\newtheorem{lemma}{Lemma}
\newtheorem{remark}{Remark}

\title{On the Design of Globally Exponentially Stable Hybrid Attitude and Gyro-bias Observers}
\author{Soulaimane Berkane, Abdelkader Abdessameud and Abdelhamid Tayebi
\thanks{This work was supported by the National Sciences and Engineering Research Council of Canada (NSERC).}
\thanks{The authors are with the Department of Electrical and Computer Engineering, University of Western Ontario, London, Ontario, Canada. A. Tayebi is also with the Department of Electrical Engineering, Lakehead University, Thunder Bay, Ontario, Canada.
     {\tt\small sberkane@uwo.ca, aabdess@uwo.ca, atayebi@lakeheadu.ca} }%
}
\begin{document}

\maketitle
\begin{abstract}
This paper presents hybrid attitude and gyro-bias observers designed directly on the Special Orthogonal group $SO(3)$. The proposed hybrid observers, enjoying global exponential stability, rely on a hysteresis-based switching between different configurations derived from a set of potential functions on $SO(3)$. Different sets of potential functions have been designed via an appropriate angular warping transformation applied to some smooth and non-smooth potential functions on $SO(3)$. We show that the proposed hybrid observers can be expressed solely in terms of inertial vector measurements and biased angular velocity readings. Simulation results are given to illustrate the effectiveness of the proposed attitude estimation approach.
\end{abstract}

\section{Introduction}
            The attitude estimation problem has generated a great deal of research work in the literature. The problem consists in recovering the attitude of a rigid body using available measurements in the body frame.
            The early attitude estimators were of a static type, designed to reconstruct the attitude from a set of vector measurements (see, for instance, \cite{Shuster1981, Markley1988}). These static attitude reconstruction techniques are hampered by their inability in handling measurement noise. To overcome this problem, researchers looked for dynamic estimators (relying on the angular velocity and inertial vector measurements) having the ability to recover the attitude while filtering measurement noise. Among these dynamic estimators, Kalman filters played a central role in aerospace applications (see, for instance, \cite{Markley2003,crassidis2007survey}).

            Recently, a new class of dynamic nonlinear attitude estimators (observers) has emerged \cite{Mahony2008}, and proved its ability in handling large rotational motions and measurement noise. This approach, coined nonlinear complementary filtering, was inspired from the linear attitude complementary filters, \textit{e.g.,} \cite{tayebi2006attitude}, used to recover (locally) the attitude using gyro and inertial vector measurements. The smooth nonlinear complementary filters, such as those proposed in \cite{Mahony2008}, are directly designed on $SO(3)$ and are proved to guarantee \textit{almost} global asymptotic stability (AGAS), which is as strong as the motion space topology could permit \cite{Bhat2000}. These smooth nonlinear observers ensure the convergence of the estimated attitude to the actual one from almost all initial conditions except from a set of critical points (equilibria) of zero Lebesgue measure. It has been noted in \cite{lee2012} that starting from a configuration close to the undesired critical points results in a slow convergence to the actual attitude. A nonlinear attitude estimator, with local stability results and improved convergence properties, has been proposed in \cite{zlotnik2017nonlinear}. In \cite{mayhew2011hybrid,lee2015tracking,berkaneCDC2015synergistic,berkane2015construction}, the topological obstruction to global asymptotic stability on compact manifolds such as $SO(3)$ has been successfully addressed via \textit{synergistic} hybrid techniques. Following this approach, a non-central hybrid attitude observer on $SO(3)$ has been proposed in \cite{lee2015observer} leading to global asymptotic stability.  Attitude estimators (evolving outside $SO(3)$) with global asymptotic and exponential stability properties have also been proposed in \cite{batista2012sensor} and \cite{Batista2012}, respectively. 

            In the present work, we develop a comprehensive approach for the design of central\footnote{The term central here refers to the use of a central family of potential functions on $SO(3)$, where all the potential functions in the family share the desired equilibrium point as a critical point. Note that in contrast with the non-central approach, each individual observer configuration derived from each  potential function in the central family, guarantees (independently) almost global asymptotic stability results.} hybrid attitude and gyro-bias observers on $SO(3)$, using biased angular velocity and inertial vector measurements, leading to global exponential stability results. First, we propose a general structure of a hybrid attitude and gyro-bias observer evolving on $SO(3)$, where the observer input depends on the gradient of some potential function on $SO(3)$ indexed by a hybrid discrete jump. We show that global exponential stability is guaranteed provided that the family of potential functions under consideration satisfies some properties. Thereafter, we propose four different methods in designing the central family of synergistic potential functions on $SO(3)$, via angular warping, enjoying the properties required for global exponential stability achievement. Two proposed estimation schemes rely on attitude information obtained using any reconstruction procedure. The other two proposed hybrid estimation schemes are explicitly expressed in terms of body-frame vector measurements. A preliminary and partial version of this work have been published in \cite{berkaneACC2016observer,berkaneCDC2016observer}. The present paper generalizes the approach and proposes different other possible designs.

\section{Background}
\subsection{Notations and Preliminaries}

        Throughout the paper, we use $\mathbb{R}$, $\mathbb{R}^+$ and $\mathbb{N}$ to denote, respectively, the sets of real, nonnegative real and natural numbers. We denote by $\mathbb{R}^n$ the $n$-dimensional Euclidean space and by $\mathbb{S}^n$ the unit $n$-sphere embedded in $\mathbb{R}^{n+1}$. We use $\|x\|$ to denote the Euclidean norm of a vector $x\in\mathbb{R}^n$. For matrices $A, B\in\mathbb{R}^{m\times n}$, the inner product is defined as $\lls A,B\ggs=\textrm{tr}(A^{\top}B)$, and the Frobenius norm of $A$ is $\|A\|_F=\sqrt{\lls A,A\ggs}$. We use $\lambda_i^A$, $\lambda_{\min}^A$ and $\lambda_{\max}^A$ to denote, respectively, the $i$-th, minimum and maximum eigenvalues of a symmetric matrix $A$.

        Let the map $[\cdot]_\times: \mathbb{R}^3\to\mathfrak{so}(3)$ be defined such that $[x]_\times y=x\times y$, for any $x, y\in\mathbb{R}^3$, where $\times$ is the vector cross-product on $\mathbb{R}^3$ and 
        $$
        \mathfrak{so}(3):=\left\{\Omega\in\mathbb{R}^{3\times 3}\mid\;\Omega^{\top}=-\Omega\right\},
        $$
is the vector space of 3-by-3 skew-symmetric matrices. Let $[\cdot]_\otimes:\mathfrak{so}(3)\to\mathbb{R}^3$ denote the inverse isomorphism of the map $[\cdot]_\times$, such that $[[\omega]_\times]_{\otimes}=\omega,$ for all $\omega\in\mathbb{R}^3$ and $[[\Omega]_{\otimes}]_\times=\Omega,$ for all $\Omega\in\mathfrak{so}(3)$. Defining $\mathbb{P}_a:\mathbb{R}^{3\times 3}\to\mathfrak{so}(3)$ as the projection map on the vector space $\mathfrak{so}(3)$ such that $ \mathbb{P}_a(A):=(A-A^{\top})/2$, one can extend the definition of $[\cdot]_\otimes$ to $\mathbb{R}^{3\times 3}$ by taking the composition map $\psi$ defined for
        any matrix $A=\{a_{ij}\}\in\mathbb{R}^{3\times 3}$ as 
        \begin{equation}\label{psi}
        \psi(A):=\Big[\mathbb{P}_a(A)\Big]_\otimes=\frac{1}{2}\left[\begin{array}{c}
        a_{32}-a_{23}\\a_{13}-a_{31}\\a_{21}-a_{12}
        \end{array}
        \right].
        \end{equation}
		The following identities are useful throughout the paper.
		\begin{align}
		\label{id::1}
		[u]_\times^2&=-\|u\|^2I+uu^\top,&\quad u\in\mathbb{R}^3,\\
		\label{id::2}
		[u\times v]_\times&=vu^\top-uv^\top,&\quad u,v\in\mathbb{R}^3,\\
		\label{id::3}
		\mathrm{tr}(uv^\top)&=u^\top v,&\quad u,v\in\mathbb{R}^3,\\
		\label{id::4}
		\lls A,[u]_\times\ggs&=\lls\mathbb{P}_a(A),[u]_\times\ggs,&\quad A\in\mathbb{R}^{3\times 3}, u\in\mathbb{R}^3,\\    
		\label{id::5}
        \lls [v]_\times,[u]_\times\ggs&=2u^\top v,&\quad v,u\in\mathbb{R}^3,\\  
        \label{id::6} 
        A[u]_\times+[u]_\times A^\top+[A^\top u]_\times&=\mathrm{tr}(A)[u]_\times,&\quad A\in\mathbb{R}^{3\times 3}, u\in\mathbb{R}^3.
        \end{align}

\subsection{Attitude Representation and Useful Relations}
        The rigid body attitude evolves on the special orthogonal group 
        $$
        SO(3) := \{ X \in \mathbb{R}^{3\times 3}|\; \mathrm{det}(X)=1,\; XX^{\top}= I \},
        $$
       where $I$ is the three-dimensional identity matrix and $X\in SO(3)$ is called a \textit{rotation matrix}. The group $SO(3)$ has a compact manifold structure with its \textit{tangent spaces} being identified by $T_XSO(3):=\left\{X\Omega\mid\Omega\in\mathfrak{so}(3)\right\}$. The inner product on $\mathbb{R}^{3\times 3}$, when restricted to the Lie algebra of $SO(3)$, defines the following \textit{left-invariant} Riemannian metric on $SO(3)$
        \begin{equation}\label{metric}
        \langle X\Omega_1,X\Omega_2\rangle_X:=\lls\Omega_1,\Omega_2\ggs,
        \end{equation}
        for all $X\in SO(3)$ and $\Omega_1,\Omega_2\in\mathfrak{so}(3)$.

        A unit quaternion\footnote{The reader is referred to \cite{Shuster1993} for more details on the unit quaternion representation.} $(\eta,\epsilon)\in\mathbb{Q}$, consists of a scalar part $\eta$ and three-dimensional vector $\epsilon$, such that $\mathbb{Q}:=\{(\eta,\epsilon)\in\mathbb{R}^4\;|\; \eta^2+\epsilon^{\top}\epsilon=1\}$.
        A unit quaternion represents a rotation matrix through the map $\mathcal{R}_Q:\mathbb{Q}\to SO(3)$ defined as
        \begin{equation}\label{rod}
        \mathcal{R}_Q(\eta,\epsilon)=I+2[\epsilon]^2_\times+2\eta[\epsilon]_\times.
        \end{equation}
        The set $\mathbb{Q}$ forms a group with the quaternion product, denoted by $\odot$, being the group operation and quaternion inverse defined by $(\eta,\epsilon)^{-1} = \left(\eta, -\epsilon \right)$ as well as the identity-quaternion $\left( 1, 0_{3\times 1} \right)$, where $0_{3\times 1}\in \mathbb{R}^3$ is a column vector of  zeros. Given $(\eta_1,\epsilon_1),~(\eta_2,\epsilon_2)\in \mathbb{Q}$, the quaternion product is defined by $(\eta_1,\epsilon_1)\odot (\eta_2,\epsilon_2) = (\eta_3, \epsilon_3)$ such that
        \begin{equation}\label{Qmultiply}
        \eta_3 = \eta_1 \eta_2- \epsilon_1^{\top}\epsilon_2, \quad \epsilon_3 = \eta_1\epsilon_2 + \eta_2\epsilon_1 +
        [\epsilon_1]_\times\epsilon_2,
        \end{equation}
        and
        \begin{equation}\label{R1R2}
        \mathcal{R}_Q(\eta_1,\epsilon_1)\mathcal{R}_Q(\eta_2,\epsilon_2)=\mathcal{R}_Q(\eta_3,\epsilon_3).
        \end{equation}
The rotation group $SO(3)$ can be also parametrized by rotations of angle $\theta\in\mathbb{R}$ around a unit-vector axis $u\in\mathbb{S}^2$. This is commonly known as the angle-axis parametrization of $SO(3)$ and is given by the map $\mathcal{R}_a:\mathbb{R}\times\mathbb{S}^2\to SO(3)$ such that
        \begin{equation}\label{Rodrigues}
        \mathcal{R}_a(\theta,u)=I+\sin(\theta)[u]_\times+(1-\cos\theta)[u]_\times^2.
        \end{equation}
        The quaternion and angle-axis representations of $SO(3)$ are related through the formulas
        \begin{align}
        \eta=\cos\left(\frac{\theta}{2}\right),\quad\epsilon=\sin\left(\frac{\theta}{2}\right)u.
        \end{align}
        For any attitude matrix $X\in SO(3)$, we define $|X|_I\in[0, 1]$ as the normalized Euclidean distance on $SO(3)$ which is given by
            \begin{equation}\label{id::normX}
            |X|_I^2:=\frac{1}{8}\|I-X\|_F^2=\frac{1}{4}\mathrm{tr}(I-X).
            \end{equation}
            The following results (proved in the Appendix section) will be useful in our subsequent analysis.
\begin{lemma}\label{lemma::derivatives}
Consider the trajectories $\dot X(t)=X(t)[u(t)]_\times$ where $X(t_0)\in SO(3)$ and $u(t)\in \mathbb{R}^3$. Then,
\begin{align}\label{eq::dtr}
\frac{d}{dt}\mathrm{tr}(A(I-X))&=2\psi(AX)^\top u,\\
\label{eq::dpsi}
\frac{d}{dt}\psi(AX)&=E(AX)u.
\end{align}
where $E(AX):=\frac{1}{2}(\mathrm{tr}(AX)-X^\top A)$.
\end{lemma}
\begin{lemma}\label{lemma::identities_Q}
  Let $A=A^\top$ and $\bar A:=\frac{1}{2}(\mathrm{tr}(A)I-A)$ be positive definite. Let $X\in SO(3)$ and $(\eta,\epsilon)\in\mathbb{Q}$ be the quaternion representation of $X$. Then, the following relations hold:
  \begin{align}
  \label{eq::Q::tr}
  \mathrm{tr}(A(I-X))&=4\epsilon^\top\bar A\epsilon,\\
  \label{eq::Q::psi}
  \psi(AX)&=2(\eta I-[\epsilon]_\times)\bar A\epsilon.
  \end{align}
\end{lemma}
\begin{lemma}\label{lemma::identities_SO(3)}
  Let $A=A^\top$ and $\bar A:=\frac{1}{2}(\mathrm{tr}(A)I-A)$ be positive definite. Then the following relations hold:
    \begin{align}\label{ineq::tr}
			4\lambda_{\min}^{\bar A}|X|_I^2&\leq\mathrm{tr}(A(I-X))\leq4\lambda_{\max}^{\bar A}|X|_I^2,\\ \label{eq::psi_norm}
			\|\psi(AX)\|^2&=\alpha(A,X)\mathrm{tr}(\underline{A}(I-X)),	\\
    		\label{ineq::alpha}
         (1-|X|_I^2)&\leq\alpha(A,X)\leq (1-\xi^2|X|_I^2),\\
         \label{ineq::E_1}
         v^\top[\lambda_{\min}^{\bar A}-E(AX)]v&\leq \frac{1}{2}\mathrm{tr}(A(I-X))\|v\|^2,\;\forall v\in\mathbb{R}^3,\\
          \|E(AX)\|_F&\leq\|\bar A\|_F,\label{ineq::E_2}
           \end{align}
            where $\alpha(A,X)=(1-|X|_I^2\cos^2(u,\bar Au))$, $\underline{A}=\mathrm{tr}(\bar A^2)I-2\bar A^2$, $\xi=\lambda_{\min}^{\bar A}/\lambda_{\max}^{\bar A}$ and $u\in\mathbb{S}^2$ is the axis of rotation $X$.
   \end{lemma}
\begin{lemma}\label{lemma::cross_prod}
            Let $A=\sum_{i=1}^{n}\rho_iv_iv_i^\top$ with $n \geq 1$, $\rho_i>0$ and $v_i\in\mathbb{R}^3, \;i=1, \ldots, n$. Then, the following hold:
            \begin{align}\label{VA_bi}
            \mathrm{tr}(A(I-XY^\top))&=\frac{1}{2}\sum_{i=1}^{n}\rho_i\|X^\top v_i-Y^\top v_i\|^2,\\\label{PA_bi}
           \psi(AXY^\top)&=\frac{1}{2}P\sum_{i=1}^{n}\rho_i (X^\top v_i\times Y^\top v_i),
            \end{align}
            for any matrices $X,Y\in SO(3)$.
\end{lemma}
\subsection{Hybrid Systems Framework}\label{appendix::hybrid}
        Let $\mathcal{M}$ be a given manifold. A general model of a hybrid system takes the form:
 \begin{eqnarray}\label{Hybrid:general}
      \left\{\begin{array}{lcl}
            \dot x\in F(x),&\quad&x\in C,\\
            x^+\in G(x), &\quad&x\in D,
            \end{array}\right.
\end{eqnarray}
       where the \textit{flow map}, $F: \mathcal{M}\to\mathsf{T}\mathcal{M}$ governs the continuous flow of $x$ on the manifold $\mathcal{M}$, the \textit{flow set} $C\subset\mathcal{M}$ dictates where the continuous flow could occur. The \textit{jump map}, $G: \mathcal{M}\to\mathcal{M}$, governs discrete jumps of the state $q$, and the \textit{jump set} $D\subset\mathcal{M}$ defines where the discrete jumps are permitted. In this paper, we consider hybrid systems written in the following form
            \begin{equation}\label{Hybrid:general2}
            \begin{array}{l}
          ~~~ \dot z=f(z,q),\\
          \left\{ \begin{array}{ll}
          \dot q=0,&\quad (z,q)\in C,\\
            q^+\in g(z,q),&\quad (z,q)\in D,
           \end{array}\right.
            \end{array}
            \end{equation}
            which is short-hand notation for a system of the form \eqref{Hybrid:general} with $x=(z,q),$ $F(x)=[f(z,q),0]$ and $G(x)=[z,g(z,q)]$.

            A subset $E\subset\mathbb{R}_{\geq 0}\times\mathbb{N}$ is a  \textit{hybrid time domain}, if it is a union of finitely or infinitely many intervals of the form $[t_j ,t_{j+1}]\times\{j\}$ where $0=t_0\leq t_1\leq t_2\leq...$,  with the last interval being possibly of the form $[t_j ,t_{j+1}]\times\{j\}$ or $[t_j ,\infty)\times\{j\}$. The ordering of points on each hybrid time domain is such that $(t,j)\preceq(t^\prime,j^\prime)$ if $t\leq t^\prime$ and $j\leq j^\prime$. A \textit{hybrid arc} is a function $\mathfrak{h}: \textrm{dom}\;\mathfrak{h}\to\mathcal{M}$, where $\textrm{dom}\;\mathfrak{h}$ is a hybrid time domain and, for each fixed $j$, $t\mapsto\mathfrak{h}(t,j)$ is a locally absolutely continuous function on the interval $I_j=\{t: (t,j)\in\textrm{dom}\;\mathfrak{h}\}$. For more details on the dynamical hybrid systems framework, the reader is referred to  \cite{Goebel2006, Goebel2009} and references therein. 
            
%
\subsection{Potential Functions on $SO(3)\times\mathcal{Q}$}
 Given a finite index set $\mathcal{Q}\subset\mathbb{N}$, let $\mathcal{C}^0\left(SO(3)\times\mathcal{Q},\mathbb{R}^+\right)$ denote the set of positive-valued functions $\Phi: SO(3)\times\mathcal{Q}\to\mathbb{R}^+$ such that for each $q\in\mathcal{Q}$, the map $R\mapsto\Phi(R,q)$ is continuous.
        If, for each $q\in\mathcal{Q}$, the map $R\mapsto\Phi(R,q)$ is differentiable on the set $D_q\subseteq SO(3)$ then the function $\Phi(R,q)$ is continuously differentiable on $\mathcal{D}\subseteq SO(3)\times\mathcal{Q}$, where $\mathcal{D}=\cup_{q\in\mathcal{Q}}D_q\times\{q\}$, in which case we denote $\Phi\in\mathcal{C}^1\left(\mathcal{D},\mathbb{R}^+\right)$.        %
        Additionally, for all $(R,q)\in\mathcal{D}$, let $\nabla\Phi(R,q)\in T_RSO(3)$ denote the gradient of $\Phi$, with respect to $R$, relative to the Riemannian metric \eqref{metric}.

        A function $\Phi\in\mathcal{C}^0\left(SO(3)\times\mathcal{Q},\mathbb{R}^+\right)$ is said to be  a \textit{potential function}
            on $\mathcal{D}\subseteq SO(3)\times\mathcal{Q}$
            with respect to the set $\mathcal{A}\subseteq\mathcal{D}$ if:
            \begin{itemize}
            \item $\Phi(R,q)>0$ for all $(R,q)\notin\mathcal{A}$,
            \item $\Phi(R,q)=0$, for all $(R,q)\in\mathcal{A}$,
            \item $\Phi\in\mathcal{C}^1(\mathcal{D},\mathbb{R}^+)$.
                    \end{itemize}
             The set of all potential functions on $\mathcal{D}$ with respect to $\mathcal{A}=\{I\}\times\mathcal{Q}$ is denoted as $\mathcal{P}_\mathcal{D}$, where a function $\Phi(R,q)\in\mathcal{P}_\mathcal{D}$ can be seen as a family of potential functions on $SO(3)$ indexed by the variable $q$. 
  \section{Problem Statement}\label{section:pb:formulation}
        Let $R\in SO(3)$ denote a rotation matrix from the body-fixed frame $\mathcal{B}$ to the inertial frame $\mathcal{I}$. The rotation matrix $R$ evolves according to the kinematics equation
        \begin{equation}\label{kinematic}
        \dot{R}=R[\omega]_\times,
        \end{equation}
        where $\omega\in\mathbb{R}^3$ is the angular velocity of the body-fixed frame $\mathcal{B}$ with respect to the inertial frame $\mathcal{I}$ expressed in the body-fixed frame $\mathcal{B}$. We suppose that a set of $n\geq 2$ vectors, denoted by $b_i$, can be measured in the body-fixed frame and are associated to a set of $n$ known inertial vectors, denoted by $a_i$, such that
            \begin{equation}\label{b_measured}
            b_i=R^{\top}a_i.
            \end{equation}
  \begin{assumption}
$n\geq 3$ body-frame vectors $b_i$ are available for measurement, and at least three of these vectors are non-collinear.
            \end{assumption}
            The vectors $b_i$ can be obtained, for example, from an inertial measurement unit (IMU) that typically includes an accelerometer and a magnetometer measuring, respectively, the gravitational field and Earth's magnetic field expressed in the body-fixed frame. Also, we suppose that the measured angular velocity, denoted by $\omega_y$, can be subject to a constant or slowly varying bias $b_\omega\in\mathbb{R}^3$ such that
            \begin{equation}\label{omega:measured}
            \omega_y=\omega+b_\omega.
            \end{equation}
            Our objective consists in designing an attitude and gyro-bias estimation algorithm, using the above described available measurements, leading to global exponential stability results.
\section{Main results}\label{section::main}
             First, we propose a general design of a hybrid attitude and gyro-bias observer depending on an indexed potential function on $SO(3)\times \mathcal{Q}$ such that $\mathcal{Q}$ is a finite index set. In particular, we show that a suitable choice of the potential function (satisfying some conditions) leads to global exponential stability. Next, we propose different methods for the design of such potential functions satisfying our derived conditions. Then, depending on the choice of the potential function, and the assumptions on the available measurements described in Section~\ref{section:pb:formulation}, we propose four different hybrid observers achieving our objective.

\subsection{Hybrid Attitude and Gyro-bias Observer Design}\label{section::hybrid}
            Let $\mathcal{Q}\subset\mathbb{N}$ be a finite index set and let $\hat{R}$ and $\hat{b}_\omega$ denote, respectively,     the estimate of the rigid body rotation matrix $R$ and the estimate of the constant bias vector $b_\omega$. Define the attitude estimation error $\tilde{R} = R\hat R^\top$ and the bias estimation error $\tilde b_\omega=b_\omega-\hat b_\omega$. We propose the following attitude and gyro-bias estimation scheme
\begin{align}\label{observer_hybrid}
            \dot{\hat R}&=\hat R\left[\omega^y-\hat b_\omega+\gamma_P\beta\big(\Phi(\tilde R,q)\big)\right]_\times,\\
            \label{bias_hybrid}
            \dot{\hat b}_\omega&=-\gamma_I\beta\big(\Phi(\tilde R,q)\big),\\
            \label{beta_hybrid}
            \beta\big(\Phi(\tilde R,q)\big)&=\hat{R}^\top\left[\tilde R^\top\nabla\Phi(\tilde R,q)\right]_\otimes,
\end{align}
            where $\hat R(t_0)\in SO(3)$, $\hat b_\omega(t_0)\in\mathbb{R}^3$, $\gamma_P$ and $\gamma_I$ are strictly positive scalars, and $\Phi\in\mathcal{P}_{\mathcal{D}}$ for some $\mathcal{D}\subseteq SO(3)\times\mathcal{Q}$. The discrete jump variable $q$ is generated by the following hybrid mechanism
                \begin{equation}\label{q}
                \left\{
                \begin{array}{rlll}
                \dot q&=&0,\;\;\;&(\tilde R,q)\in\mathcal{F},\\
                q^+&\in&\mathrm{arg}\underset{p\in{\mathcal{Q}}}{\mathrm{min}}\;\Phi(\tilde R,p),\;\;\;&(\tilde R,q)\in\mathcal{J},
                \end{array}\right.
                \end{equation}
                where the flow set $\mathcal{F}$ and jump set $\mathcal{J}$ are defined by
                \begin{eqnarray}
                \label{F}\mathcal{F}&=&\{(\tilde R,q): \Phi(\tilde R,q)-\underset{p\in{\mathcal{Q}}}{\mathrm{min}}\;\Phi(\tilde R,p)\leq\delta\},\\
                \label{J}\mathcal{J}&=&\{(\tilde R,q): \Phi(\tilde R,q)-\underset{p\in{\mathcal{Q}}}{\mathrm{min}}\;\Phi(\tilde R,p)\geq\delta\},
                \end{eqnarray}
                for some $\delta>0$. A necessary condition to implement this hybrid estimation scheme is that $\mathcal{F}\subseteq\mathcal{D}$.

\begin{theorem}\label{theorem1}
            Consider the attitude kinematics \eqref{kinematic} coupled with the observer \eqref{observer_hybrid}-\eqref{J}. Assume that the potential function $\Phi\in\mathcal{P}_{\mathcal{D}}$, for some $\mathcal{D}\subseteq SO(3)\times\mathcal{Q}$, and  the hysteresis gap $\delta>0$ are chosen such that $\mathcal{F}\subseteq\mathcal{D}$ and
\begin{align}
            \label{ineq::Phi}
            \alpha_1|\tilde R|_I^2&\leq\Phi(\tilde R,q)\leq\alpha_2|\tilde R|_I^2,\quad\forall~(\tilde R,q)\in SO(3)\times\mathcal{Q},\\\label{ineq::DPhi}
            \alpha_3|\tilde R|_I^2&\leq\|\nabla\Phi(\tilde R,q)\|_F^2\leq\alpha_4|\tilde R|_I^2,\quad\forall~(\tilde R,q)\in\mathcal{F},
            \end{align}
            where $\alpha_i>0,i=1,\ldots, 4$ are strictly positive scalars.
            Assume, in addition, that the angular velocity $\omega(t)$ is uniformly bounded. Then, the number of discrete jumps is finite and the equilibrium point $e = 0$, with $e:=(|\tilde R|_I,\|\tilde b_\omega\|)^\top$, is uniformly globally exponentially stable.
\end{theorem}

\begin{proof}
In view of \eqref{kinematic}, \eqref{omega:measured} and \eqref{observer_hybrid}-\eqref{bias_hybrid}, one obtains
        \begin{align*}
        \dot{\tilde R}&=\dot R\hat R^\top-R\hat R^\top\dot{\hat R}\hat R^\top\\
        					&=R[\omega]_\times\hat R^\top-R\left[\omega^y-\hat b_\omega+\gamma_P\beta\big(\Phi(\tilde R,q)\big)\right]_\times\hat R^\top\\
        					&=\tilde R\left[\hat R\big(-\tilde b_\omega-\gamma_P\beta\big(\Phi(\tilde R,q)\big)\big)\right]_\times,
        \end{align*}
        where the fact that $[u]_\times P^\top=P^\top[Pu]_\times$, for all $u\in\mathbb{R}^3$ and $P\in SO(3)$, has been used to obtain the last equality. The closed loop dynamics during the flows of $\mathcal{F}\times\mathbb{R}^3$ are given then by
        \begin{align}\label{tilde R}
            \dot{\tilde R}&=\tilde R\left[\hat R\big(-\tilde b_\omega-\gamma_P\beta\big(\Phi(\tilde R,q)\big)\big)\right]_\times,\\
               \dot q&=0,\\\label{tilde b}
            \dot{\tilde b}_\omega&=\gamma_I\beta\big(\Phi(\tilde R,q)\big)\big),
        \end{align}

        First, we show that $\tilde b_\omega$ is bounded. To this end, we consider the following real-valued function on $SO(3)\times\mathcal{Q}\times\mathbb{R}^3$
        \begin{align}\label{L0}
         \mathfrak{L}_0(\tilde R, q, \tilde b_\omega)=\Phi(\tilde R,q)+\frac{1}{\gamma_I}\|\tilde b_\omega\|^2,
        \end{align}
        which is positive definite with respect to $$\bar{\mathcal{A}}:=\{(\tilde R, q, \tilde b_\omega)\in SO(3)\times\mathcal{Q}\times \mathbb{R}^3\mid \tilde R=I, \tilde{b}_\omega = 0\}.$$ 
        Now making use of \eqref{id::4}-\eqref{id::5}, the time derivative of $\mathfrak{L}_0$, along the trajectories of \eqref{tilde R}-\eqref{tilde b}, can be shown to satisfy
           \begin{align}
            \dot{\mathfrak{L}}_0(\tilde R,q,\tilde b_\omega)=&\lls\nabla\Phi(\tilde R,q),\tilde R\left[\hat R\Big(-\tilde b_\omega-\gamma_P\beta\big(\Phi(\tilde R,q)\big)\Big)\right]_\times\ggs\nonumber\\
            &+ \frac{2}{\gamma_I}\tilde b_\omega^\top\left(\gamma_I\beta\big(\Phi(\tilde R,q)\big)\right)\nonumber\\
            =&-\gamma_P\lls\nabla\Phi(\tilde R,q),\tilde R\Big[\hat R\beta\big(\Phi(\tilde R,q)\big)\Big]_\times\ggs\nonumber\\
            &-2\Big[\tilde R^\top\nabla\Phi(\tilde R,q)\Big]_\otimes^\top\hat R\tilde b_\omega+2\tilde b_\omega^\top\beta\big(\Phi(\tilde R,q)\big)\nonumber\\
           =&-\gamma_P\|\nabla\Phi(\tilde R,q)\|_F^2\leq0,\label{dL0}
           \end{align}
            for all $(\tilde{R}, q,\tilde b_\omega)\in\mathcal{F}\times\mathbb{R}^3$, where \eqref{beta_hybrid} has been used to obtain the last equality. Therefore, $\mathfrak{L}_0$ is non-increasing along the flows of $\mathcal{F}\times\mathbb{R}^3$. Also, for all $(\tilde R,q)\in\mathcal{J}$ and $q^+\in\mathrm{arg}\underset{p\in{\mathcal{Q}}}{\mathrm{min}}\;\Phi(\tilde R,p)$, one has
            \begin{align}
            \mathfrak{L}_0(\tilde R, q^+,\tilde b_\omega)&-\mathfrak{L}_0(\tilde R,q,\tilde b_\omega)=~\Phi(\tilde R,q^+)-\Phi(\tilde R,q)\leq-\delta,\label{diffL0}
            \end{align}
            which implies that $\mathfrak{L}_0$ is strictly decreasing over the jumps. Consequently, the set $\bar{\mathcal{A}}$ is stable by \cite[Theorem 7.6]{Teel2007}. Since the positive definite function $\mathfrak{L}_0$ is non-increasing, every solution is bounded preventing escape time. It should be noted that, since jumps map the state to $\mathcal{F}\setminus\mathcal{J}$, it follows from \cite[Proposition 2.4]{Goebel2006} that every solution is complete\footnote{ A solution  $\mathfrak{h}(t,j)$ to a hybrid system is complete if $\mathrm{dom} \mathfrak{h}$ is unbounded \cite{teel2013lyapunov}, {\it i.e.}, there is an infinite number of jumps and/or the continuous time is infinite.}.

Now, consider the Lyapunov function candidate
            \begin{equation}\label{Lypa:pf:Thm1}
            \mathfrak{L}(\tilde R,\hat R, q,\tilde b_\omega)=\mathfrak{L}_0(\tilde R, q, \tilde b_\omega)+\mu\tilde b_\omega^\top\hat R^\top\psi(\tilde R),
            \end{equation}
           where $\mathfrak{L}_0$ is given in \eqref{L0} and $\mu>0$. Using \eqref{ineq::tr}-\eqref{eq::psi_norm}, it can be verified  that
           \begin{align}\label{ineq::psi}
           \|\psi(\tilde R)\|^2=4|\tilde R|_I^2\big(1-|\tilde R|_I^2\big).
           \end{align}
            Hence, \eqref{ineq::Phi} and \eqref{ineq::psi} can be used to show that the function $\mathfrak{L}$ satisfies the quadratic inequality
                \begin{equation}\label{ineq_L}
                e^\top P_1 e \leq \mathfrak{L} \leq e^\top P_2 e,
                \end{equation}
            where $e=\big(|\tilde R|_I,\|\tilde b_\omega\|\big)^\top$ and
            $$P_1=\begin{bmatrix}
                        \alpha_1&-\mu\\ -\mu&1/\gamma_I
                    \end{bmatrix}, \qquad
                    P_2=\begin{bmatrix}
                        \alpha_2&\mu\\ \mu&1/\gamma_I
                    \end{bmatrix}.$$
            The matrices $P_1$ and $P_2$ are positive definite provided that $0<\mu<\sqrt{\alpha_1/\gamma_I}$ and, in this case, the function $\mathfrak{L}$ is positive definite with respect to the equilibrium $e=0$.

            Now, we need to evaluate the time derivative of $\mathfrak{L}$ along the flows of $\mathcal{F}\times\mathbb{R}^3$.   Making use of \eqref{eq::dpsi} with \eqref{observer_hybrid} and \eqref{tilde R}-\eqref{tilde b}, the time derivative of $\mathfrak{X}(\tilde R,\hat R,\tilde b_\omega):=\tilde b_\omega^\top\hat R^\top\psi(\tilde R)$ along the flows of $\mathcal{F}\times\mathbb{R}^3$ is obtained as
            \begin{align}
                \dot{\mathfrak{X}}(\tilde R,\hat R,\tilde b_\omega)=&\tilde b_\omega^\top\hat R^\top\dot\psi(\tilde R)-\tilde b_\omega^\top[\omega-\tilde b_\omega+\gamma_P\beta]_\times\hat R^\top\psi(\tilde R)\nonumber +\dot{\tilde b}_\omega^\top\hat R^\top\psi(\tilde R)\nonumber\\
                =&~\tilde b_\omega^\top\hat R^\top E(\tilde R)\left(\hat R(-\tilde b_\omega-\gamma_P\beta)\right)\nonumber\\
                &~-\tilde b_\omega^\top[\omega+\gamma_P\beta]_\times\hat R^\top\psi(\tilde R)+\gamma_I\beta^\top\hat R^\top\psi(\tilde R)\nonumber,
            \end{align}
            where the arguments of $\beta$ have been omitted for simplicity. Let $c_\omega :=\sup_{t\geq 0}\omega(t)$ and $c_b:=\sup_{(t,j)\succeq (t_0,0)}\|\tilde b_\omega(t,j)\|$, which exist, respectively, in view of our assumption on $\omega(t)$ and the fact that all solutions are bounded. One can verify from \eqref{beta_hybrid} that
           $$\|\beta\|^2=\frac{1}{2}\|\nabla\Phi(\tilde R,q)\|_F^2\leq\frac{1}{2}\alpha_4|\tilde R|_I^2.$$
            Therefore, in view of inequalities \eqref{ineq::E_1}-\eqref{ineq::E_2}, \eqref{ineq::psi} and the above results, the time derivative of the cross term $\mathfrak{X}(\tilde R,\hat R,\tilde b_\omega)$ during the flows of $\mathcal{F}\times\mathbb{R}^3$ satisfies
            \begin{align}
                \dot{\mathfrak{X}}(\tilde R,\hat R,\tilde b_\omega) \leq&~-\|\tilde b_\omega\|^2+2|\tilde R|_I^2\|\tilde b_\omega\|^2+\gamma_P\sqrt{3}\|\tilde b_\omega\|\|\beta\|\nonumber\\
                &~+\gamma_I\|\beta\|\big\|\psi(\tilde R)\big\|+c_\omega\|\tilde b_\omega\|\big\|\psi(\tilde R)\big\|	\nonumber\\ \nonumber
                &~+\gamma_Pc_b\|\beta\|\big\|\psi(\tilde R)\big\|\\
                \leq&~ e^\top P_\mathfrak{X} e,\label{dot_frak_X}
            \end{align}
            with the matrix $P_\mathfrak{X}$ being defined as

        $$P_{\mathfrak{X}}=\begin{bmatrix}
       \sqrt{2\alpha_4}(\gamma_I+c_b\gamma_P)+2c_b^2&c_\omega+\gamma_P\sqrt{\frac{3\alpha_4}{8}}\\
        c_\omega+\gamma_P\sqrt{\frac{3\alpha_4}{8}}&-1
        \end{bmatrix}.$$

        Now, making use of inequality \eqref{dot_frak_X}, the expression of the time derivative of $\mathfrak{L}_0$ in \eqref{dL0} along with inequality \eqref{ineq::DPhi}, one can show that the time-derivative of $\mathfrak{L}$ in \eqref{Lypa:pf:Thm1} satisfies
        \begin{equation}
            \dot{\mathfrak{L}}\leq -e^\top
            \underbrace{\left( \begin{bmatrix}
            		\gamma_P\alpha_3&0\\0&0
            				\end{bmatrix}-\mu P_\mathfrak{X}\right)}_{P_3}e,\label{dL}
            \end{equation}
        for all $(\tilde{R}, q,\tilde b_\omega)\in\mathcal{F}\times\mathbb{R}^3$. To guarantee that the matrices $P_1,~P_2$ and $P_3$ are positive definite, it is sufficient to pick $\mu$ such that
        \begin{equation*}
        0<\mu<\min\left\{\frac{\sqrt{\alpha_1}}{\sqrt{\gamma_I}},\frac{\gamma_P\alpha_3}{P_{\mathfrak{X}11}+P_{\mathfrak{X}12}^2}\right\},
        \end{equation*}
        where $P_{\mathfrak{X}11}= \sqrt{2\alpha_4}(\gamma_I+c_b\gamma_P)+2c_b^2$ and $P_{\mathfrak{X}12}= c_\omega+\gamma_P\sqrt{3\alpha_4/8}$.

        Then, using \eqref{dL} with \eqref{ineq_L}, we conclude that $\dot{\mathfrak{L}}\leq-\lambda_F\mathfrak{L}$ with $\lambda_F:= \lambda_{\min}^{P_3}/\lambda_{\max}^{P_2}$; $\mathfrak{L}$ is exponentially decreasing along the flows of $\mathcal{F}$. Equivalently, one has\footnote{ Note that for the sake of presentation simplicity, we used  $\mathfrak{L}(t,j)$ to denote $\mathfrak{L}\big(\tilde R(t,j),\hat R(t,j),q(t,j),\tilde b_\omega(t,j)\big)$. }
         \begin{eqnarray}
        \label{L1}
        \mathfrak{L}(t,j)&\leq& e^{-\lambda_F(t-t^\prime)}\mathfrak{L}(t^\prime,j),
        \end{eqnarray}
        for all $(t,j),~(t^\prime,j)\in\mathrm{dom}~(\tilde R,q, \tilde{b}_{\omega})$ with $(t,j)\succeq (t^\prime,j)$.
        Furthermore, it can verified from \eqref{Lypa:pf:Thm1} and \eqref{diffL0} that
        \begin{equation}\label{diffL}\mathfrak{L}(\tilde R, \hat R, q^+,\tilde b_\omega)-\mathfrak{L}(\tilde R,\hat R,q,\tilde b_\omega)
        \leq -\delta,
        \end{equation}
        for all $(\tilde R,q)\in\mathcal{J}$ and $q^+\in\mathrm{arg}\underset{p\in{\mathcal{Q}}}{\mathrm{min}}\;\Phi(\tilde R,p)$. Consequently, one can conclude that  $\mathfrak{L}$ is strictly decreasing over the jumps of $\mathcal{J}$. Equivalently, for all $(t,j)\in\mathrm{dom}~(\tilde R,q, \tilde{b}_{\omega})$ such that $(t,j+1)\in\mathrm{dom}~(\tilde R,q, \tilde{b}_{\omega})$, one has
        \begin{equation}
           \label{L2}
           \mathfrak{L}(t,j+1)-\mathfrak{L}(t,j)\leq -\delta.
        \end{equation}
        It is clear from \eqref{ineq_L}, \eqref{L1} and \eqref{L2}, that the results of the theorem are trivial in the case where there is no discrete jumps, \textit{i.e.,} $j=0$. Therefore, in the remainder of this proof, we will consider only the case where $j\geq 1$.
        From  (\ref{L1})-(\ref{L2}), one can easily show that
        \begin{eqnarray}
        \label{L3}
        0< \mathfrak{L}(t,j) \leq \mathfrak{L}(t_0,0)-\delta j,
        \end{eqnarray}
        which leads to
        \begin{eqnarray}
        \label{L4}
        j \leq \frac{\mathfrak{L}(t_0,0)-\mathfrak{L}(t,j)}{\delta} <  \frac{\mathfrak{L}(t_0,0)}{\delta}.
        \end{eqnarray}
This shows that the number of jumps is finite. Since the solution is complete and the number of jumps is bounded, the hybrid time domain of the solution takes the form $\mathrm{dom}~(\tilde R,q, \tilde{b}_{\omega})=\cup_{j=0}^{j_{\max}-1}\big([t_j,t_{j+1}]\times\{j\}\big)\cup[t_{j_{\max}},+\infty)\times\{j_{\max}\}$ where $j_{\max}$ denotes the maximum number of discrete jumps.

         Now, one can show from \eqref{L2} that
        \begin{equation}
        \label{L5}
        \mathfrak{L}(t,j+1)\leq\left(1-\frac{\delta}{\mathfrak{L}(t_0,0)}\right)\mathfrak{L}(t,j)\leq e^{-\sigma}\mathfrak{L}(t,j),
        \end{equation}
        for all $(t,j)\in\mathrm{dom}~(\tilde R,q, \tilde{b}_{\omega})$ such that $(t,j+1)\in\mathrm{dom}~(\tilde R,q, \tilde{b}_{\omega})$, where $\sigma=-\ln\left(1-\frac{\delta}{\mathfrak{L}(t_0,0)}\right)$. Note that $\delta<\mathfrak{L}(t_0,0)$ 
        as per \eqref{L4} (and $j\geq 1$).
        Now, we will show, by induction, that the following inequality holds for all $ j\geq 1$
        \begin{eqnarray}
        \label{L6}
        \mathfrak{L}(t,j)\leq e^{-\lambda(t-t_0+j)}\mathfrak{L}(t_{0},0),
        \end{eqnarray}
        with $\lambda=\min\{\lambda_F, \sigma\}$. Using \eqref{L1} and \eqref{L5}, one can easily show that (\ref{L6}) is satisfied for $j=1$ as follows. Assume that $(t,1)\in\mathrm{dom}~(\tilde R,q, \tilde{b}_{\omega})$ then
        \[
        \begin{array}{rcl}
        \mathfrak{L}(t,1)&\leq& e^{-\lambda_F(t-t_{1})}\mathfrak{L}(t_{1},1)\\
        ~&\leq& e^{-\lambda_F(t-t_{1})} e^{-\sigma}\mathfrak{L}(t_{1},0)\\
        ~&\leq& e^{-\lambda(t-t_{1})} e^{-\lambda}e^{-\lambda(t_1-t_{0})}\mathfrak{L}(t_{0},0)\\
        ~&\leq& e^{-\lambda(t-t_0+1)}\mathfrak{L}(t_{0},0).
        \end{array}
        \]
        Assuming that (\ref{L6}) holds true for $j=k$, and using (\ref{L1}) and (\ref{L5}), one has
        \[
        \begin{array}{rcl}
        \mathfrak{L}(t,k+1)&\leq& e^{-\lambda_F(t-t_{k+1})}\mathfrak{L}(t_{k+1},k+1)\\
        ~&\leq& e^{-\lambda_F(t-t_{k+1})} e^{-\sigma}\mathfrak{L}(t_{k+1},k)\\
        ~&\leq& e^{-\lambda(t-t_{k+1})} e^{-\lambda}e^{-\lambda(t_{k+1}-t_0+k)}\mathfrak{L}(t_0,0)\\
        ~&\leq& e^{-\lambda(t-t_0+k+1)}\mathfrak{L}(t_0,0),
        \end{array}
        \]
        for $(t,k+1)\in\mathrm{dom}~(\tilde R,q, \tilde{b}_{\omega})$. Therefore, inequality (\ref{L6}) also holds for $j=k+1$ and hence it holds true for all $j\geq 1$.
        Finally, in view of (\ref{L6}) and \eqref{ineq_L}, one can conclude that

        \begin{equation}\label{estimate:error:thm1}
                |e(t,j)|_I^2 \leq k_e e^{-\lambda(t-t_0+j)} |e(t_0, 0)|_I^2,
            \end{equation}
        for all $(t,j)\in \mathrm{dom}~(\tilde{R}, q, \tilde{b}_\omega)$, where  $k_e = \lambda_{\max}^{P_2}/\lambda_{\min}^{P_1}$.
%
        The proof is complete.
\end{proof}
        Theorem \ref{theorem1} provides sufficient conditions on the potential function $\Phi$ ensuring that the hybrid attitude and gyro-bias estimation scheme \eqref{observer_hybrid}-\eqref{J} guarantees global exponential stability of the estimation errors. It can be noticed that the estimator \eqref{observer_hybrid}-\eqref{J} relies on the rotation matrix $\tilde{R}$ which is not directly available for feedback. The choice of the potential function $\Phi$ with the parameters of the hybrid mechanism in \eqref{q}-\eqref{J}, and the implementation of the proposed observer using the available measurements will be discussed in detail in the next subsections.

        Before this, we consider the following two special cases that require some modification of the estimation algorithm in Theorem~\ref{theorem1}. In the case where it is required to guarantee {\it a priori} bounded bias estimates, which is generally desirable in adaptive control algorithms, the proposed estimation algorithm can be modified using a projection mechanism provided that the unknown bias $b_\omega$ satisfies $\|b_\omega\| \leq \bar{b}_\omega$ for some known $\bar{b}_\omega$.  In particular, the adaptation law \eqref{bias_hybrid} can be replaced by
        \begin{equation}\label{bias_hybrid_projection}
         \dot{\hat b}_\omega=\mathrm{Proj}\left(-\gamma_I\beta(\Phi(\tilde R,q)\big),\hat b_\omega\right),
        \end{equation}
        where $\mathrm{Proj}(\mu,\hat b_\omega)=p(\hat b_\omega)\mu$, with $p(\hat b_\omega)=I$ if $\|\hat b_\omega\|\leq\bar b_\omega$ or $\hat b_\omega^\top\mu\leq 0$, otherwise $p(\hat b_\omega)=I-\hat b_\omega\hat b_\omega^\top/\|\hat b_\omega\|^2$, and satisfies the following properties \cite{marino1998robust}:\vspace{0.01 in}
            \begin{itemize}
            \item [(P1)] $\|\hat b_\omega(t,j)\|\leq\bar b_\omega,\;\forall (t,j)\succeq (t_0,0)$,
            \item [(P2)] $(\hat b_\omega-b_\omega)^\top\mathrm{Proj}(\mu,\hat b_\omega)\leq(\hat b_\omega-b_\omega)^\top\mu,$, 
            \item [(P3)] $\|\mathrm{Proj}(\mu,\hat b_\omega)\|\leq\|\mu\|$.
            \end{itemize}
\begin{cor}\label{corollary::Projection}
         Consider system \eqref{kinematic} with the observer \eqref{observer_hybrid} with \eqref{bias_hybrid_projection} and \eqref{beta_hybrid}-\eqref{J}, where $\|b_\omega\| \leq \bar{b}_\omega$. Let the potential function $\Phi$ and $\delta$ be selected as in Theorem~\ref{theorem1}
          and assume that the angular velocity $\omega(t)$ is uniformly bounded.
          Then, the equilibrium point $e = 0$, with $e:=(|\tilde R|_I,\|\hat b_\omega - b_\omega\|)^\top$, is uniformly globally exponentially stable and the number of discrete jumps is bounded.
\end{cor}
        \begin{proof}
        Consider the Lyapunov function candidate $\mathfrak{L}$, given in \eqref{Lypa:pf:Thm1} with \eqref{L0}, and satisfies \eqref{ineq_L}. Exploiting the properties P1-P3 and following similar steps in \eqref{dL0}-\eqref{diffL0} and \eqref{dot_frak_X}, one can show that the time-derivative of $\mathfrak{L}$ satisfies \eqref{dL} and \eqref{diffL} with $c_b = 2\bar{b}_\omega$. Then, the result of the corollary follows using the same arguments after \eqref{dL} in the proof of Theorem~\ref{theorem1}.
        \end{proof}

        In the case where the angular velocity measurements are given by \eqref{omega:measured} with $b_\omega \equiv 0$ \textit{i.e.,}, unbiased angular velocity measurements, the following corollary of Theorem~\ref{theorem1} can be shown.
\begin{cor}\label{corollary::exponential bound}
         Consider the attitude kinematics \eqref{kinematic} coupled with the observer \eqref{observer_hybrid} with  \eqref{beta_hybrid}-\eqref{J} and $\hat{b}_\omega \equiv 0$. Let the potential function $\Phi$ and $\delta$ be selected as in Theorem~\ref{theorem1}. Then, the number of discrete jumps is bounded and the equilibrium point $|\tilde R|_I = 0$ is uniformly globally exponentially stable. More precisely,
        \begin{equation}\label{estimate_error_cor2}
        |\tilde R(t,j)|_I^2\leq\frac{\alpha_2}{\alpha_1}e^{-\frac{\gamma_P\alpha_3}{\alpha_2}(t-t_0)}|\tilde R(t_0,0)|_I^2-\frac{\delta}{\alpha_1}\sum_{s=1}^{j}e^{-\frac{\gamma_P\alpha_3}{\alpha_2}(t-t_s)},
        \end{equation}
        for all $(t,j)\in\mathrm{dom}(\tilde R,q)$, where $\alpha_i$, $i=1,...,3$ are given in \eqref{ineq::Phi}-\eqref{ineq::DPhi}. 
\end{cor}
\begin{proof}
            Consider the Lyapunov function candidate $\bar{\mathfrak{L}}=\Phi(\tilde R,q)$. Following similar steps as in \eqref{L0}-\eqref{diffL0} in the proof of Theorem~\ref{theorem1}, the time-derivative of $\bar{\mathfrak{L}}$ is obtained as
            \begin{align}\label{dot:L:pf:cor2}
                \dot{\bar{\mathfrak{L}}}\leq-\gamma_P\alpha_3|\tilde R|_I^2\leq-\lambda_F\Phi(\tilde R,q),
            \end{align}
            with $\lambda_F=\gamma_P(\alpha_3/\alpha_2)$, during the flows of $\mathcal{F}$. Hence, with the definition of $\bar{\mathfrak{L}}$, one has
            \begin{align}\label{estimate:Phi:pf:cor2}
                \Phi(\tilde R(t,j),q(t,j))\leq e^{-\lambda_F(t-t^\prime)}\Phi(\tilde R(t^\prime,j),q(t^\prime,j)),
            \end{align}
            for all $(t,j),(t^\prime,j)\in\mathrm{dom}~(\tilde R,q)$ such that $(t,j)\succeq(t^\prime,j)$.  In addition, during the jumps of $\mathcal{J}$, one has
            \begin{align}\label{diff:L:pf:cor2}
            \Phi(\tilde R(t,j+1),q(t,j+1))-\Phi(\tilde R(t,j),q(t,j))\leq -\delta,
            \end{align}
             for all $(t,j)\in\mathrm{dom}~(\tilde R,q)$ such that $(t,j+1)\in\mathrm{dom}~(\tilde R,q)$.
Following similar steps as in the proof of Theorem \ref{theorem1}, it can be verified that the solution is complete, the number of jumps is bounded and the hybrid time domain of the solution takes the form $\mathrm{dom}~(\tilde R,q)=\cup_{j=0}^{j_{\max}-1}\big([t_j,t_{j+1}]\times\{j\}\big)\cup[t_{j_{\max}},+\infty)\times\{j_{\max}\}$ where $j_{\max}$ denotes the maximum number of discrete jumps. The global exponential stability of $|\tilde R|_I=0$ also follows using similar arguments as in the proof of Theorem \ref{theorem1}.

In the case where there is no discrete jump \textit{i.e.,} $j=0$, the bound \eqref{estimate_error_cor2} follows directly from \eqref{estimate:Phi:pf:cor2} and \eqref{ineq::Phi}. Note that the operator $\sum_{s=1}^j$ is understood to be zero when $j=0$. Now, let us consider the case when $j\geq 1$ and show, by induction, that the following inequality\footnote{ Also for the sake of presentation simplicity, we use $\Phi(t,j)$ to denote $\Phi\big(\tilde R(t,j),q(t,j)\big)$. } holds for all $j\geq 1$
\begin{equation}\label{Phi1}
\Phi(t,j)\leq e^{-\lambda_F(t-t_0)}\Phi(t_0,0)-\delta\sum_{s=1}^{j}e^{-\lambda_F(t-t_s)}.
\end{equation}
It is not difficult to show that \eqref{Phi1} holds for $j=1$. In fact, in view of \eqref{diff:L:pf:cor2}-\eqref{estimate:Phi:pf:cor2} and for all $(t,1)\in\mathrm{dom}~(\tilde R,q)$, one obtains
            \[
            \begin{array}{rcl}
             \Phi(t,1)&\leq&e^{-\lambda_F(t-t_1)}\Phi(t_1,1)\\
             &\leq&e^{-\lambda_F(t-t_1)}\big(\Phi(t_1,0)-\delta\big)\\
             &\leq&e^{-\lambda_F(t-t_1)}\big(e^{-\lambda_F(t_1-t_0)}\Phi(t_0,0)-\delta\big)\\
             &\leq&e^{-\lambda_F(t-t_0)}\Phi(t_0,0)-\delta e^{-\lambda_F(t-t_1)},
            \end{array}
            \]
which shows that inequality \eqref{Phi1} holds for $j=1$. Assuming that \eqref{Phi1} holds for $j=k$, and using \eqref{diff:L:pf:cor2}-\eqref{estimate:Phi:pf:cor2}, one has for $(t,k+1)\in\mathrm{dom}~(\tilde R,q)$
          \[
            \begin{array}{rcl}
             \Phi(t,k+1)&\leq&e^{-\lambda_F(t-t_{k+1})}\Phi(t_{k+1},k+1)\\
           &\leq&e^{-\lambda_F(t-t_{k+1})}\big(\Phi(t_{k+1},k)-\delta\big)\\
            &\leq&e^{-\lambda_F(t-t_{k+1})}\big(e^{-\lambda_F(t_{k+1}-t_0)}\Phi(t_0,0)\\
            &&\hspace{1.5cm}-\delta\sum_{s=1}^ke^{-\lambda_F(t_{k+1}-t_s)}-\delta\big)
            \\
           &\leq&e^{-\lambda_F(t-t_0)}\Phi(t_0,0)-\delta\sum_{s=1}^{k+1}e^{-\lambda_F(t-t_s)},
\end{array}
            \]
            which shows that \eqref{Phi1} also holds for $j=k+1$ and, hence, it holds true for all $j\geq 1$. Finally, making use of \eqref{ineq::Phi} and \eqref{Phi1}, inequality \eqref{estimate_error_cor2} follows.
\end{proof}
\begin{remark}
        The bound obtained in \eqref{estimate_error_cor2} depends on the properties of the potential function $\Phi$, especially the coefficients $\alpha_1,\alpha_2$ and $\alpha_3$, the observer gain $\gamma_P$ and the hysteresis gap $\delta$. It should be mentioned here that a similar estimate of the error vector $e$ in Theorem~\ref{theorem1} can be derived, however, it would depend on the unknown eigenvalues of matrix $P_3$ in \eqref{dL}.
\end{remark}
\subsection{Construction of the Potential Function $\Phi$}\label{section::potential}
            In this subsection, we construct potential functions $\Phi$, along with the hysteresis gap $\delta>0$, satisfying our assumptions in Theorem~\ref{theorem1}.
\vspace{0.05 in}
\subsubsection{Traditional potential functions}\label{section::traditional}
        Consider the following potential function on $SO(3)\times\{1\}$: 
                \begin{align}\label{UA}
        \Phi_s(\tilde R,1)=U_A(\tilde R)&:=\mathrm{tr}\big(A(I-\tilde R)\big)/4\lambda_{\max}^{\bar A},
        \end{align}
        where $A=A^\top$ such that $\bar{A}:=\frac{1}{2}(\mathrm{tr}(A)I-A)$ is positive definite. Note that the \textit{smooth} function $U_A$ has been widely used in attitude control systems design \cite{Koditschek,Mahony2008,Sanyal2009,berkane2015construction}. In view of \eqref{ineq::tr}, it can be seen that $\Phi_s(\tilde R,1)$ satisfies
            \begin{equation}\label{ineq::UA}
            \xi|\tilde R|_I^2\leq \Phi_s(\tilde R,1)\leq|\tilde R|_I^2,
            \end{equation}
        where $\xi:=\lambda_{\min}^{\bar A}/\lambda_{\max}^{\bar A}$ and, hence, $\Phi_s(\tilde R,1)$ satisfies the first condition \eqref{ineq::Phi} of Theorem \ref{theorem1}. However, it can be shown that $\Phi_s(\tilde R,1)$ does not satisfy \eqref{ineq::DPhi} in Theorem \ref{theorem1}. In particular, the relation $\alpha_3|\tilde R|_I^2\leq\|\nabla\Phi_s(\tilde R,1)\|_F^2,\;\alpha_3>0$, requires that $\nabla\Phi_s(\tilde R,1)$ does not vanish on the flow set $\mathcal{F}=SO(3)\times\{1\}$ except at $\tilde R = I$, which is not the case. In fact,
        the gradient of the potential function $\Phi_s(\tilde R,1)$ is given by (see \cite[Lemma 2]{berkane2015construction})
        \begin{equation}\label{dUA}
        \nabla\Phi_s(\tilde R,1)=\tilde R\mathbb{P}_a(A\tilde R)/4\lambda_{\max}^{\bar A},
        \end{equation}
        which, in view of \eqref{id::4}-\eqref{id::5} and \eqref{eq::psi_norm}, satisfies
        \begin{align*}
        \|\nabla\Phi_s(\tilde R,1)\|_F^2&= \frac{1}{16(\lambda_{\max}^{\bar A})^2}\lls\mathbb{P}_a(A\tilde R),\mathbb{P}_a(A\tilde R)\ggs\\
        												&= \frac{1}{8(\lambda_{\max}^{\bar A})^2}\|\psi(A\tilde R)\|^2,\\
        												&= \frac{1}{8(\lambda_{\max}^{\bar A})^2}\epsilon^\top\bar A^2\epsilon\big(1-\|\epsilon\|^2\cos^2(\phi)\big),
        \end{align*}
        where $\phi:=\angle(\epsilon,\bar A\epsilon)$ and $\epsilon$ is the vector part of the unit quaternion corresponding to the attitude matrix $\tilde R$. Consequently, the subset of $SO(3)\times\{1\}$ where $\nabla\Phi_s(\tilde R,1)=0$, called also the set of critical points of $\Phi_s(\tilde R,1)$, corresponds to $\mathcal{S}_I\times\{1\}$ and $\mathcal{S}_\pi\times\{1\}$ where
        \begin{equation}
        \label{S_I}\mathcal{S}_I:=\{\tilde R\in SO(3)\mid\tilde R=I\},\end{equation}
        \begin{equation}\label{S_pi}\mathcal{S}_\pi:=\{\tilde R\in SO(3)\mid\tilde R=\mathcal{R}_Q(0,v), v\in\mathcal{E}(A)\},
        \end{equation}
            with $\mathcal{E}(A)$ being the set of all eigenvectors of $A$.

        Therefore, the appearance of the undesired equilibrium points in $\mathcal{S}_\pi$ cannot be avoided when using an attitude observer design based on the gradient of $U_A$, as done in \cite{Mahony2008} with some matrix $A$. In addition, performance degradation (slow convergence) is reported in this case for large attitude errors \cite{lee2015observer,zlotnik2017nonlinear}. An alternate approach to design a gradient-based attitude observer on $SO(3)$ is to consider the following non-differentiable function 
        \begin{align}
        \Phi_{ns}(\tilde R,1)=V_A(\tilde R)&:=2\left(1-\sqrt{1-U_A(\tilde R)}\right).\label{VA}
        \end{align}
        Since $U_A(\tilde R)\in[0,1]$ and using \eqref{ineq::UA}, it is easy to verify that
        \begin{align}\label{ineq::VA}
        \xi|\tilde R|_I^2\leq U_A(\tilde R)\leq V_A(\tilde R)\leq2U_A(\tilde R)\leq 2|\tilde R|_I^2,
        \end{align}
        and hence $\Phi_{ns}(\tilde R,1)$ is also quadratic with respect to $|\tilde R|_I^2$; $\Phi_{ns}(\tilde R,1)$ satisfies \eqref{ineq::Phi} in Theorem~\ref{theorem1}. Note that we consider in \eqref{VA} a weighted version of the function $V_I$, obtained from \eqref{VA} with $A=I$, used in \cite{lee2012} where it has been shown that control systems designed based on $V_I$ exhibit faster convergence rates for large attitude manoeuvres as compared to those designed using the smooth function $U_A$. The potential function $V_I$ was also shown to be the solution to the kinematic optimal control problem on $SO(3)$ in \cite{Saccon2010}. However, gradient-based control systems based on $V_I$ are less frequent in the literature, as compared to those obtained from $U_A$, due to its non-differentiability for rotations of angle $180^{\mathrm{o}}$. In fact, it can be verified from  \eqref{VA} that
        \begin{equation}\label{dVA}
        \nabla\Phi_{ns}(\tilde R,1)=\frac{\nabla\Phi_s(\tilde R,1)}{\sqrt{1-\Phi_s(\tilde R,1)}},
        \end{equation}
        which is not defined for all $(\tilde R,1)$ satisfying $\Phi_s(\tilde R,1)=U_A(\tilde R)=1$. This corresponds to the set $ \mathcal{S}_{\pi,\max}\times\{1\}$, with
        \begin{align}
        \mathcal{S}_{\pi,\max}:=\{\tilde R\in SO(3)\mid~&\tilde R=\mathcal{R}_Q(0,v),\nonumber\\
        \label{S_pi_max}&v\in\mathcal{E}(A),\;\bar Av=\lambda_{\max}^{\bar A}v\},
        \end{align}
        containing the singular points of $\Phi_{ns}(\tilde{R}, 1)$. Consequently, $\Phi_{ns}(\tilde{R}, 1)\in\mathcal{P}_{\mathcal{D}_{ns}}(\mathcal{A})$, with $\mathcal{D}_{ns}=\big(SO(3)\setminus \mathcal{S}_{\pi, \max}\big)\times \{1\}$, and, in addition, the critical points of $\Phi_{ns}(\tilde R,1)$ are contained in $\big(\mathcal{S}_{\pi}\setminus\mathcal{S}_{\pi,\max}\big)\times\{1\}$, with the set $\mathcal{S}_{\pi}$ given in \eqref{S_pi}. Note that $\mathcal{S}_{\pi,\max}\subseteq\mathcal{S}_{\pi}$; In particular, $\mathcal{S}_{\pi,\max} = \mathcal{S}_{\pi}$ for $A=I$. It can be verified then that $\Phi_{ns}(\tilde{R}, 1)$ cannot satisfy \eqref{ineq::DPhi} in Theorem~\ref{theorem1} for all $(\tilde{R}, 1)\in\mathcal{F}$, where $\mathcal{F} = SO(3)\times \{1\}$ in this case.

        Even though the functions $\Phi_s(\tilde{R}, 1)$ and $\Phi_{ns}(\tilde{R}, 1)$ do not satisfy the conditions of Theorem~\ref{theorem1}, they can both be used in the design of appropriate potential functions satisfying our requirements through an adequate transformation described in the following.
\vspace{0.05 in}
\subsubsection{Angular warping}\label{sub:sub:subsection:warping}
        To construct potential functions satisfying the conditions of Theorem \ref{theorem1}, we introduce the following \textit{angular warping} transformation \cite{berkane2015construction}
        \begin{align}\label{Gamma}
        \Gamma_A(\tilde R,q)&=\tilde R\mathfrak{R}_A(\tilde R,q)\\\label{Rq}
        \mathfrak{R}_A(\tilde R,q)&=\mathcal{R}_a\left(2\sin^{-1}(kU_A(\tilde R)),\nu(q)\right),
        \end{align}
        where $U_A$ is defined as in \eqref{UA} for some matrix $A=A^\top$ such that $\bar{A} =\frac{1}{2}\mathrm{tr}(A)I-A$ is positive definite, $q\in\mathcal{Q}$ for some index set $\mathcal{Q}\subset \mathbb{N}$, the map $\nu(q): \mathcal{Q}\to\mathbb{S}^2$ to be determined, and the scalar $k$ satisfies
        \begin{equation}\label{k_max}
        0<k<\bar k:=\frac{1}{\sqrt{6-\max\{1,4\xi^2\}}},\qquad \xi:=\frac{\lambda_{\min}^{\bar A}}{\lambda_{\max}^{\bar A}}.
        \end{equation}
        The transformation $\Gamma_A(\tilde R,q)$ can be regarded as a perturbation of $\tilde R$ about the unit vector $\nu(q)$ by an angle $2\sin^{-1}(kU_A(\tilde R))$. The above condition on the scalar $k$ guarantees that the map $\tilde R\to\Gamma_A(\tilde R,q)$ is everywhere a local diffeomorphism \cite{berkane2015construction}. We recall the following Lemma which can be derived from \cite[Lemma 1 \& 3]{berkane2015construction}. 
\begin{lemma}
For any $X\in SO(3)$ and $\omega\in\mathbb{R}^3$ satisfying $\dot X=X[\omega]_\times$, one has
        \begin{equation}\label{dG}
        \frac{d}{dt}\Gamma_A(X,q)=\Gamma_A(X,q)\left[\Theta_A(X,q)\omega\right]_\times,
        \end{equation}
        where the matrix $\Theta_A(X,q)$ is full rank and is given by
        \begin{align}\label{Theta}
        \Theta_A(X,q)&=\mathfrak{R}_A(X,q)^\top+\frac{4k\nu(q)\psi(AX)^\top}{4\lambda_{\max}^{\bar A}\sqrt{1-k^2U_A^2(X)}}.
        \end{align}
\end{lemma}
Moreover, some useful properties of the transformation $\Gamma_A$ are given in the following lemma proved in Appendix \ref{proof:lemma:conidition1}.
  \begin{lemma}\label{lemma::condition1}
        Consider the transformation $\Gamma_A$ in \eqref{Gamma}-\eqref{k_max}.
        Then,
        \begin{equation}\label{upper:lower:bounds:Gamma}
        \underline{\gamma}|\tilde R|_I^2\leq|\Gamma_A(\tilde R,q)|_I^2\leq \overline{\gamma}|\tilde R|_I^2,
        \end{equation}
        for all $(\tilde R,q)\in SO(3)\times\mathcal{Q}$
        with $\bar\gamma,\underline{\gamma}>0$ given by  
        \begin{equation}\label{bounds:Gamma}
        \underline{\gamma}:= 1-k^2-k\sqrt{1-k^2}, \qquad \overline{\gamma}:= 1+k+\frac{k^2}{4}.
        \end{equation}
       
        \end{lemma}
        It should be noted that $\underline{\gamma}$ in \eqref{upper:lower:bounds:Gamma}-\eqref{bounds:Gamma} is strictly positive under condition \eqref{k_max}. Inequality \eqref{upper:lower:bounds:Gamma} shows that the transformation $\Gamma_A(\tilde R,q)$ acts
on the attitude distance on $SO(3)$, namely the norm $|\tilde{R}|_{I}$, in a way such that the new attitude $\Gamma_{A}(\tilde{R},q)$ has a distance, namely $|\Gamma_{A}(\tilde{R},q)|_{I}$, which is bounded from below and above by a term proportional to the original attitude distance $|\tilde{R}|_{I}$. 

        We show in the following that the potential functions $U_A\big(\Gamma_A(\tilde{R}, q)\big)$ and $V_A\big(\Gamma_A(\tilde{R}, q)\big)$ satisfy the conditions of Theorem \ref{theorem1} through an appropriate choice of matrix $A$, the index set $\mathcal{Q}$, the map $\nu$, and the hysteresis gap $\delta$ in \eqref{q}-\eqref{J}.
\subsubsection{Composite potential functions}
        Define
        \begin{equation}\label{Phi_UA_VA}\begin{array}{lcl}
        \Phi_{U_A}(\tilde{R}, q) &=& U_A \circ \Gamma_A(\tilde{R}, q),\\
         \Phi_{V_A}(\tilde{R}, q) &=& V_A \circ \Gamma_A(\tilde{R}, q),\end{array}
        \end{equation}
        where $U_A$ and $V_A$ are given in \eqref{UA} and \eqref{VA}, respectively, and $\Gamma_A$ is defined in Section~\ref{sub:sub:subsection:warping}. The main motivation behind introducing the modified functions in \eqref{Phi_UA_VA} is to avoid the critical/singular points of the functions $U_A$ and $V_A$ using the transformation $\Gamma_A$ via an appropriate design of the switching mechanism \eqref{q}-\eqref{J}. In fact, the transformation $\Gamma_A(\tilde R,q)$ allows to stretch and compress the manifold $SO(3)$ by moving all the points (except the identity rotation $I$) to different locations; In particular, for each index $q\in\mathcal{Q}$, the transformation  $\Gamma_A(\tilde R,q)$ allows to re-locate all the points of the set $\mathcal{S}_\pi$ given in \eqref{S_pi}.

        Consider the following possible designs of the parameters in \eqref{q}-\eqref{J} and \eqref{Gamma}-\eqref{Rq}:
\begin{itemize}
        \item [D1.] $\Phi=\Phi_{U_A}$, $A=I$, $\mathcal{Q}\in\{1,\cdots,6\}$, $\nu(p)=e_p,\;p\in\{1,2,3\},\;\nu(p+3)=-\nu(p)$ and $0<\delta<\Delta_{I}(k):=(-1+\sqrt{1+4k^2})^{3}/24k^4$; where $\{e_1,e_2,e_3\}$ is any orthonormal basis on $\mathbb{R}^3$ and the scalar $k$ satisfies \eqref{k_max}.
        \item [D2.] $\Phi=\Phi_{V_A}$, $A=I$, $\mathcal{Q}$ and $\nu(\cdot)$ are defined as in D1, and the hysteresis gap $\delta$ satisfies $0<\delta<\Delta_{II}(k):=2\sqrt{\Delta_{I}(k)}$, with $k$ satisfying \eqref{k_max}.
        \item [D3.] $\Phi=\Phi_{U_A}$, $A$ is positive definite with the distinct eigenvalues $0<\lambda_1<\lambda_2<\lambda_3$, $\mathcal{Q}\in\{1,2\}$, $\nu(1)=u,\;\nu(2)=-u$, where
            the vector $u\in\mathbb{S}^2$ satisfies:
                \begin{equation}\label{choice_u_1}u^\top v_{1}=0,\qquad (u^\top v_{i})^2=\lambda^{A}_i/(\lambda^{A}_2+\lambda^{A}_3),\end{equation}
                for $i\in\{1,2\}$ if $ \lambda_2^{A}\lambda_3^{A}-\lambda_1^{A} \lambda_2^{A}-\lambda_1^{A}\lambda_3^{A}\geq 0$, or
                \begin{equation}\label{choice_u_2}(u^\top v_{i})^2=1-\frac{4\prod_{j\neq i}\lambda_j^{A}}{\sum_\ell\sum_{k\neq \ell}\lambda_\ell^{A}\lambda_k^{A}},\end{equation}
            otherwise, for $i\in\{1,2,3\}$ and $v_i$ being the eigenvector of $A$ corresponding to the eigenvalue $\lambda_i^A$. The hysteresis gap $\delta$ satisfies $0<\delta<\Delta_{III}(k):=4k^2\bar V^2(1-k^2\bar V^2)\Lambda$   with $k$ selected as in \eqref{k_max}, $\bar V=[-1+\sqrt{1+4k^2\xi\Lambda}]/2k^2\Lambda$, and
            \begin{equation}\nonumber\Lambda :=\left\{ \begin{array}{ll} \lambda_1^A/(\lambda_2^A+\lambda_3^A) & \mbox{if}~
                    \lambda_2^{A}\lambda_3^{A}-\lambda_1^{A} \lambda_2^{A}-\lambda_1^{A}\lambda_3^{A}\geq 0,\\
                    \frac{4\prod_j\lambda_j^{A}}{(\lambda_2^A+\lambda_3^A)\sum_\ell\sum_{k\neq \ell}\lambda_\ell^{A}\lambda_k^{A}} &\mbox{otherwise,}
                    \end{array}\right.
               \end{equation}
            \item [D4.] $\Phi=\Phi_{V_A}$, $A$, $\mathcal{Q}$, and $\nu(\cdot)$ are given as in D3, and the hysteresis gap satisfies $0<\delta<\Delta_{IV}(A,k):=2[-\sqrt{1-\xi}+\sqrt{1-\xi+\Delta_{III}(A,k)}]$, with $k$ and $\xi$ given in \eqref{k_max}.
\end{itemize}
        Making use of one of the designs above, the following result, proved in Appendix~\ref{proof::lemma::synergism}, can be deduced.
        \begin{lemma}\label{lemma::synergism}
        Consider the functions $\Phi_{U_A}$ and $\Phi_{V_A}$ in \eqref{Phi_UA_VA} with the transformation $\Gamma_A(\tilde{R}, q)$ being defined in \eqref{Gamma}-\eqref{k_max} and the discrete variable $q$ satisfying \eqref{q}-\eqref{J}. Suppose that $\Phi$, $\mathcal{Q}$, $\delta$  in \eqref{q}-\eqref{J} as well as matrix $A$ and the map $\nu$ in \eqref{Gamma}-\eqref{Rq} are selected according to one of the designs D1 -- D4. Then,
        for all $(\tilde R,q)\in\mathcal{F}$, one has $\Gamma_A(\tilde R,q)\notin \mathcal{S}_{\pi}$ where $\mathcal{F}$ is given in \eqref{F} and $\mathcal{S}_{\pi}$ is defined in \eqref{S_pi}. Moreover, for each of the designs in D1 -- D4, we have $\Phi\in\mathcal{P}_{\mathcal{D}}$ and $\mathcal{F}\subseteq \mathcal{D}$, with $\mathcal{D}$ being defined in the proof below.
        \end{lemma}
        Lemma~\ref{lemma::synergism} indicates that $\Gamma_A(\tilde R,q)\notin \mathcal{S}_{\pi}$ can be guaranteed for all $(\tilde{R}, q)\in\mathcal{F}$ if one considers one of the potential functions in \eqref{Phi_UA_VA} with an appropriate choice of the design parameters. This property is crucial in the design of hybrid observers ensuring global stability results.

        Now, we show that the presented potential functions in Lemma~\ref{lemma::synergism} satisfy \eqref{ineq::Phi}-\eqref{ineq::DPhi} in Theorem~\ref{theorem1}. For this, we need to compute the gradient of each potential function. For any $X\in SO(3)$, and $\omega\in\mathbb{R}^3$ satisfying $\dot X=X[\omega]_\times$, one can show using \eqref{id::4}-\eqref{id::5}, \eqref{dUA}, and \eqref{dG}, that
        \begin{align*}
            &\frac{d}{dt}\Phi_{U_A}(X,q)\\
            &=\lls\nabla U_A(\Gamma_A(X,q)),\dot{\Gamma}_A(X,q)\ggs\\
            &=\frac{1}{4\lambda_{\max}^{\bar A}}\lls\mathbb{P}_a(A\Gamma_A(X,q)),\left[\Theta_A(X,q)\omega\right]_\times\ggs\\
            &=\frac{1}{2\lambda_{\max}^{\bar A}}\psi(A\Gamma_A(X,q))^\top\Theta_A(X,q)\omega\\
            &=\frac{1}{4\lambda_{\max}^{\bar A}}\lls X[\Theta_A(X,q)^\top\psi(A\Gamma_A(X,q))]_\times,X[\omega]_\times\ggs .
            \end{align*}
            On the other hand, the gradient of $\Phi_{U_A}$ verifies
                \begin{align*}
                \frac{d}{dt}\Phi_{U_A}(X,q)&=\lls\nabla\Phi_{U_A}(X,q),\dot{X}\ggs.
                \end{align*}
            Consequently,
                \begin{equation}\label{DUAG}
                \nabla\Phi_{U_A}(X,q)=\frac{1}{4\lambda_{\max}^{\bar A}}X\left[\Theta_A(X,q)^\top\psi(A\Gamma_A(X,q))\right]_\times.
                \end{equation}
       In view of \eqref{VA} and \eqref{Phi_UA_VA}, and applying the chain rule, one has
                \begin{equation}\label{DVAG}
                \nabla\Phi_{V_A}(X,q)=\frac{\nabla\Phi_{U_A}(X,q)}{\sqrt{1-\Phi_{U_A}(X,q)}}.
                \end{equation}

        It should be noted that the expression of the gradient $ \nabla\Phi_{V_A}$ in \eqref{DVAG} is well defined on the flow set $\mathcal{F}$ due to the fact that $\Gamma_A(X,q)\notin\mathcal{S}_\pi$ for all $(X,q)\in\mathcal{F}$ (see Lemma~\ref{lemma::synergism}).
        \begin{prop}\label{proposition::condition2}
        Consider the functions $\Phi_{U_A}$ and $\Phi_{V_A}$ in \eqref{Phi_UA_VA} with the transformation $\Gamma_A(\tilde{R}, q)$ being defined in \eqref{Gamma}-\eqref{k_max} and the discrete variable $q$ satisfying \eqref{q}-\eqref{J}.
        Suppose that $\Phi$, $\mathcal{Q}$, $\delta$  in \eqref{q}-\eqref{J} as well as matrix $A$ and the map $\nu$ in \eqref{Gamma}-\eqref{Rq} are selected according to one of the designs D1 -- D4. Then the conditions \eqref{ineq::Phi}-\eqref{ineq::DPhi} in Theorem~\ref{theorem1} are satisfied for the potential function $\Phi_{U_A}$ with
                \begin{align*}
                &\alpha_1=\xi\underline{\gamma},\qquad \qquad \alpha_2=\overline{\gamma},\\
                &\alpha_3=\frac{\underline{\lambda}\xi^2\underline{\gamma}}{4}\min_{(\tilde R,q)\in\mathcal{F}}\left(1-|\Gamma_A(\tilde R,q)|_I^2\cos^2(\phi_q)\right), \quad
                \alpha_4=\frac{\overline{\lambda}\overline{\gamma}}{4},
                \end{align*}
                and for the potential function $\Phi_{V_A}$ with
                \begin{align*}
                &\alpha_1=\xi\underline{\gamma},\qquad \qquad \alpha_2=2\overline{\gamma},\\
                &\alpha_3=\frac{\underline{\lambda}\xi^2\underline{\gamma}}{4}\min_{(\tilde R,q)\in\mathcal{F}}\left(\frac{1-|\Gamma_A(\tilde R,q)|_I^2\cos^2(\phi_q)}{1-\Phi_{U_A}(\tilde R,q)}\right),\\
                &\alpha_4=\frac{\overline{\lambda}\overline{\gamma}}{4}\max_{(\tilde R,q)\in\mathcal{F}}\left(\frac{1-|\Gamma_A(\tilde R,q)|_I^2\cos^2(\phi_q)}{1-\Phi_{U_A}(\tilde R,q)}\right),
                \end{align*}
                where $(\underline{\lambda}, \overline{\lambda})$ are given in the proof below, $(\underline{\gamma}, \overline{\gamma})$ are defined in \eqref{bounds:Gamma}, $\phi_q=\angle(\epsilon_q,\bar A\epsilon_q)$, and $\epsilon_q$ is the vector part of the unit quaternion corresponding to the rotation $\Gamma_A(\tilde R,q)$.
\end{prop}
\begin{proof}
            First, using \eqref{ineq::UA} (respectively \eqref{ineq::VA}) and the results of Lemma \ref{lemma::condition1}, it is straightforward to show that $\Phi_{U_A}$ (respectively $\Phi_{V_A}$) satisfies \eqref{ineq::Phi} with the corresponding $\alpha_1$ and $\alpha_2$ given in the Proposition.

            Now, for $(\tilde R,q)\in SO(3)\times\mathcal{Q}$, let $\lambda^{\Theta}_{\min}(\tilde R, q)$ and $\lambda^{\Theta}_{\max}(\tilde R, q)$ denote, respectively, the smallest and largest eigenvalue of $\Theta_A(\tilde R,q)\Theta_A(\tilde R,q)^{\top}$, and let the constants $\displaystyle\underline{\lambda} = \min_{SO(3)\times \mathcal{Q}}(\lambda^{\Theta}_{\min}(\tilde R, q))$ and $\displaystyle\overline{\lambda}= \max_{SO(3)\times \mathcal{Q}}(\lambda^{\Theta}_{\max}(\tilde R, q))$. It is clear that $\underline{\lambda},~\overline{\lambda}>0$ by the fact that $\Theta_A(\tilde R,q)$ is full rank. Then, from \eqref{DUAG}, one can show that
{\small
        \begin{align*}
        \|\nabla\Phi_{U_A}(\tilde R,q)\|_F^2&=\frac{1}{4(\lambda_{\max}^{\bar A})^2}\lls\tilde R\left[\Theta_A(\tilde R,q)\psi(A\Gamma_A(\tilde R,q))\right]_\times,\\
        &\hspace{1.8cm}\tilde R\left[\Theta_A(\tilde R,q)\psi(A\Gamma_A(\tilde R,q))\right]_\times\ggs
        \\
        &=\frac{1}{4(\lambda_{\max}^{\bar A})^2}\left\|\Theta_A(\tilde R,q)\psi(A\Gamma_A(\tilde R,q))\right\|^2\\
        														&\leq\frac{\overline{\lambda}}{4(\lambda_{\max}^{\bar A})^2}\left\|\psi(A\Gamma_A(\tilde R,q))\right\|^2\\
        														&\leq\frac{1}{4}\overline{\lambda}|\Gamma_A(\tilde R,q)|_I^2\leq\frac{\overline{\lambda}\overline{\gamma}}{4}|\tilde R|_I^2,
        \end{align*}
        }
        where relations \eqref{eq::psi_norm}-\eqref{ineq::alpha} and the result of Lemma \ref{lemma::condition1} have been used. Also,
{\small
         \begin{align*}
            \|\nabla\Phi_{U_A}(\tilde R,q)\|_F^2&=\frac{1}{4(\lambda_{\max}^{\bar A})^2}\left\|\Theta_A(\tilde R,q)\psi(A\Gamma_A(\tilde R,q))\right\|^2\\										&\geq\frac{\underline{\lambda}}{4(\lambda_{\max}^{\bar A})^2}\left\|\psi(A\Gamma_A(\tilde R,q))\right\|^2\\
           &\geq\frac{\underline{\lambda}(\lambda_{\min}^{\bar A})^2}{4(\lambda_{\max}^{\bar A})^2}|\Gamma_A(\tilde R,q)|_I^2\Big(1-\\
           &\hspace{3cm}|\Gamma_A(\tilde R,q)|_I^2\cos^2(\phi_q)\Big).
        \end{align*}
        }
     with  $\phi_q=\angle(\epsilon_q,\bar A\epsilon_q)$ and $\epsilon_q$ is the vector part of the unit quaternion corresponding to $\Gamma_A(\tilde R,q)$. However, in view of Lemma~\ref{lemma::synergism}, one has $\Gamma_A(\tilde R,q)\notin \mathcal{S}_{\pi}$  for all $(\tilde R,q)\in\mathcal{F}$ where $v$ is an eigenvector of $\bar A$ which implies that $\epsilon_q$ can not align with $\bar A\epsilon_q$ when $|\Gamma_A(\tilde R,q)|_I^2=1$.
        This implies that
          $$
          1-|\Gamma_A(\tilde R,q)|_I^2\cos^2(\phi_q)>0,
          $$
        during the flows of $\mathcal{F}$. On the other hand, in view of \eqref{VA}, the gradient of the potential function $\Phi_{V_A}$ satisfies \eqref{DVAG}.
        Therefore, in view of the above obtained results and using the result of Lemma \ref{lemma::condition1}, the result of Proposition \ref{proposition::condition2} follows.
\end{proof}
         Proposition \ref{proposition::condition2}, with Lemma~\ref{lemma::synergism}, provide several methods for the design of potential functions satisfying the conditions in Theorem~\ref{theorem1}. Each of the proposed potential functions can be used for the design of the hybrid attitude and gyro-bias observer presented in Subsection~\ref{section::hybrid} leading to global exponential stability results. This is summarised as follows:
\begin{theorem}\label{theorem:summary}
        Consider the attitude kinematics \eqref{kinematic} coupled with the observer \eqref{observer_hybrid}-\eqref{J} and let the potential function $\Phi(\tilde{R}, q)$ and the hysteresis gap $\delta$ be selected as in Proposition \ref{proposition::condition2} and, accordingly, $\nabla\Phi(\tilde{R}, q)$ in \eqref{beta_hybrid} is determined from \eqref{DUAG} or \eqref{DVAG}. Suppose that the angular velocity $\omega(t)$ is uniformly bounded. Then, the results of Theorem~\ref{theorem1} hold.
\end{theorem}
\begin{remark}
         The results of Corollary~\ref{corollary::Projection} and Corollary~\ref{corollary::exponential bound} can also be shown to hold when using the potential function $\Phi(\tilde{R}, q)$ and the hysteresis gap $\delta$ in Theorem~\ref{theorem:summary}.
\end{remark}
        \begin{remark}\label{remark:diff:U or V}
         Despite the fact that a similar stability result is guaranteed in all design cases D1 -- D4, it is important to mention some differences between setting $\Phi = \Phi_{U_A}$ or $\Phi = \Phi_{V_A}$ (for any $A$ in Lemma~\ref{lemma::synergism}). These differences mainly reside on the parameters $\alpha_i$, $i=1,\ldots,4$, obtained in Proposition~\ref{proposition::condition2} in each case. For example, the scalar $\alpha_3$ associated to $\Phi_{V_A}$ is increased by a factor $1/(1-\Phi_{U_A})$ compared to the one associated to $\Phi_{U_A}$. It can be seen from \eqref{estimate_error_cor2} for instance that a change in $\alpha_3$ will influence the convergence rate of the attitude estimation error.
\end{remark}

\subsection{Implementation of the Hybrid Observers using Inertial Measurements}
            The hybrid schemes described in the previous section depend explicitly on $\tilde{R}=R\hat{R}^\top$, and hence on $R$ which is not directly measured. Note that it is possible to algebraically reconstruct the attitude matrix $R$ from the available inertial measurements, described in Section~\ref{section:pb:formulation}, using static attitude determination algorithms \cite{Shuster1981, Markley1988}. The resulting estimation scheme can be seen as a filtering algorithm for the reconstructed attitude matrix. It is desirable in practice, however, to use directly the vector measurements in the estimation algorithm without reconstructing the rotation matrix. In our preliminary result \cite{berkaneACC2016observer}, we have shown that the hybrid observer in Theorem~\ref{theorem:summary}, with the parameters selected according to the second design method D2, (named Observer II) can be written using some new vectors defined as a combination of the measured inertial vectors. While the method in \cite{berkaneACC2016observer} is attractive in the sense that it allows for explicit expressions of the observer in terms of the newly defined vectors, it yet requires an important preconditioning process of the measured vectors (see \cite[Lemma 2]{berkaneACC2016observer}).

            In this section, we present explicit formulations of the hybrid attitude and gyro-bias observers obtained using designs D3 and D4 in Theorem~\ref{theorem:summary} using directly the measurements of inertial vectors satisfying Assumption 1. Let the matrix $A$ be defined as
            \begin{equation}\label{A:for:Theorem3}
            A = \sum_{i=1}^{n}\rho_ia_ia_i^\top,
            \end{equation} for some $\rho_i>0$, and let $\bar{A}:=\mathrm{tr}(A)I-A$. Define also the following quantities
                \begin{equation}\label{R:vertheta:Theta:explicit}\begin{array}{l}
                 \bar\Phi=\displaystyle\frac{1}{8\lambda_{\max}^{\bar A}}\displaystyle\sum_{i=1}^{n}\rho_i\|b_i-\hat R^\top\bar{\mathfrak{R}}a_i\|^2,\\
                   \bar\beta=\displaystyle\frac{1}{8\lambda_{\max}^{\bar A}}\hat R^\top\bar\Theta\hat R\displaystyle\sum_{i=1}^{n}\rho_i(b_i\times\hat R^\top\bar{\mathfrak{R}}a_i),\\
                \bar{\mathfrak{R}}=\displaystyle\mathcal{R}_a(2\sin^{-1}(k\vartheta),\nu(q)),\\
                \vartheta=\displaystyle\sum_{i=1}^{n}\rho_i\|b_i-\hat R^\top a_i\|^2/8\lambda_{\max}^{\bar A},\\
                \bar\Theta =\displaystyle\Big(I+\frac{k\hat R\sum_{i=1}^{n}\rho_i(b_i\times\hat R^\top a_i)\nu(q)^\top}{2\lambda_{\max}^{\bar A}\sqrt{1-\vartheta^2}}\Big).
                \end{array}
                \end{equation}
Our result in this subsection is given in the following theorem.
        \begin{theorem}\label{proposition::explicit}
        Consider the four hybrid observers obtained from \eqref{observer_hybrid}-\eqref{J} with the potential function $\Phi$, the index set $\mathcal{Q}$, the map $\nu(\cdot)$, and the hysteresis gap $\delta$ being selected according to one of designs D3--D4, where $\Phi_{U_A}$, $\Phi_{V_A}$, $\Gamma_A$ are given in \eqref{Phi_UA_VA} and \eqref{Gamma}-\eqref{k_max}. Also, let the matrix $A$ used in \eqref{Phi_UA_VA} and \eqref{Gamma} be defined as in \eqref{A:for:Theorem3} with $\rho_i$ selected such that $A$ has distinct eigenvalues.
        Suppose also that Assumption 1 on the vector measurements is satisfied. Then, the terms $\Phi(\tilde{R}, q)$ and $\beta(\Phi(\tilde{R}, q))$ in \eqref{observer_hybrid}-\eqref{J} can be written in terms of the measured vectors, for each observer, as follows:
                \begin{align}\label{phi_beta_UA_explicit}
                \mbox{Observer III}:&~\left\{\begin{array}{l}
                   \Phi(\tilde R,q)=\bar\Phi,\\
                   \beta(\Phi(\tilde R,q))=\bar\beta,
                    \end{array}\right.\\
                \label{phi_beta_VA_explicit}
                \mbox{Observer IV}:&~\left\{\begin{array}{l}
                   \Phi(\tilde R,q)=2\big(1-\displaystyle\sqrt{1-\bar\Phi}\big),\\
                   \beta(\Phi(\tilde R,q))=\frac{\bar\beta}{\sqrt{1-\bar\Phi}},
                    \end{array}\right.
                    \end{align}
        where $\bar{\Phi}$, $\bar{\beta}$ are defined in \eqref{R:vertheta:Theta:explicit}. In addition, both hybrid observers III and IV, when coupled with system \eqref{kinematic}, ensures the result of Theorem~\ref{theorem1} provided that the angular velocity of the rigid body is uniformly bounded.
            \end{theorem}
        \begin{proof}
 Let us show that relation \eqref{phi_beta_UA_explicit} holds for the positive-definite matrix $A$ as defined in the theorem. Consider the smooth potential function $U_A$ given in \eqref{UA}. Using the result in Lemma~\ref{lemma::cross_prod}, one can easily deduce that
        \begin{align}\label{UA::explicit}
            U_A(\tilde R)&=\frac{1}{8\lambda_{\max}^{\bar A}}\sum_{i=1}^{n}\rho_i\|b_i-\hat R^\top a_i\|^2,\\
            \psi(A\tilde R)&=\frac{1}{2}\hat R\sum_{i=1}^{n}\rho_i(b_i\times\hat R^\top a_i),\label{psiR::explicit}
        \end{align}
        where we used $\tilde{R} = R \hat{R}^\top$ and $b_i = R^\top a_i$, $i=1,\ldots, n$. It can also be deduced from Lemma \ref{lemma::cross_prod} and \eqref{Phi_UA_VA} with \eqref{Gamma} that
        \begin{align}\label{Phi_UA::explicit:inproof}
            \Phi_{U_A}(\tilde R,q)&=\frac{1}{8\lambda_{\max}^{\bar A}}\sum_{i=1}^{n}\rho_i\|b_i-\hat R^\top\mathfrak{R}_A(\tilde R,q)a_i\|^2,\\
            \label{Psi:phiUA:explicit:inproof}\psi(\Gamma_A(\tilde R,q))&=\frac{1}{2}\mathfrak{R}_A(\tilde R,q)^\top\hat R\sum_{i=1}^{n}\rho_i(b_i\times\hat R^\top\mathfrak{R}_A(\tilde R,q)a_i),
        \end{align}
        with $\mathfrak{R}_A(\tilde R,q)$ given in \eqref{Rq}. In addition, relations \eqref{beta_hybrid} with \eqref{DUAG} can be used to show that
        \begin{equation}\label{beta3}
            \beta(\Phi_{U_A}(\tilde R,q))=\frac{1}{4\lambda_{\max}^{\bar A}}\hat R^\top\Theta_A(\tilde R,q)\psi(A\Gamma(\tilde R,q)).
        \end{equation}
        Then, \eqref{phi_beta_UA_explicit} with \eqref{R:vertheta:Theta:explicit} can be obtained by taking into account \eqref{Rq} and \eqref{Theta} with \eqref{UA::explicit}-\eqref{Psi:phiUA:explicit:inproof}.   It should be noted from Lemma~\ref{lemma::synergism} and \eqref{DVAG} that
        $$\beta\big(\Phi_{V_A}(\tilde{R}, q)\big)=\beta\big(\Phi_{U_A}(\tilde{R}, q)\big)\frac{1}{\sqrt{1-\Phi_{U_A}(\tilde R,q)}}$$
        is well defined on the flows of $\mathcal{F}$. In addition, one can verify from \eqref{R:vertheta:Theta:explicit} and \eqref{UA::explicit} that $\vartheta = k U_A(\tilde{R})$ which, together with \eqref{UA}, imply that $\vartheta^2 < 1$ holds under condition \eqref{k_max}.

        To show that the hybrid observers described in the theorem ensure the uniform global exponential stability result in Theorem~\ref{theorem1}, it is enough to verify the conditions of Theorem~\ref{theorem1}. Consider Observer III (respectively, Observer IV) from \eqref{observer_hybrid}-\eqref{J} with \eqref{R:vertheta:Theta:explicit}-\eqref{phi_beta_UA_explicit}. In addition to the fact that $A=\sum_{i=1}^{n}\rho_i a_i a_i^\top$ in this case is positive definite, it is possible always to select the weighting scalars $\rho_i>0$ such that the eigenvalues of $A$ are distinct. With such a choice, the parameters  $\Phi$, $A$, $\mathcal{Q}$, $\nu(\cdot)$ and $\delta$ correspond to design method D3 (respectively D4). Therefore, the result of the theorem can be shown, for each hybrid observer, using the results of Lemma~\ref{lemma::synergism}, Proposition \ref{proposition::condition2}, and Theorem~\ref{theorem1}. The proof is complete.
\end{proof}

\begin{remark}
        Explicit formulation of the hybrid observers satisfying the results of Corollaries~\ref{corollary::Projection} and \ref{corollary::exponential bound} can also be derived following the same lines in Theorem~\ref{proposition::explicit}.
\end{remark}
        Theorem~\ref{proposition::explicit} provides explicit formulations of two hybrid attitude and gyro-bias observers, in terms of the available inertial vector measurements. It can be noticed that only two configurations are needed for Observers III and IV. It is important to notice that this Assumption 1 is only technical and does not exclude the case where measurements of only two vectors are available, say $b_1$ and $b_2$ corresponding to the non-collinear inertial vectors $a_1$ and $a_2$. In this case, one can always construct a third vector $b_3=b_1\times b_2$ which corresponds to the measurement of $a_3 = a_1\times a_2$. The differences between Observers III and IV mentioned in Remark~\ref{remark:diff:U or V} will be further studied through numerical examples.

        In \cite{Mahony2008}, an attitude observer of the form \eqref{observer_hybrid}-\eqref{bias_hybrid} has been proposed with an input $\beta$ being selected as the sum of the vector-errors between the measured vectors $b_i$ and their estimates $\hat R^\top a_i$, {\it i.e.}, $\beta=\sum_{i=1}^n\rho_i(b_i\times\hat R^\top a_i)$. This smooth observer can be obtained from \eqref{phi_beta_UA_explicit}, with \eqref{R:vertheta:Theta:explicit}, by setting $k=0$ and choosing the parameters corresponding to Observer III. As mentioned above, the corresponding hybrid attitude and gyro-bias observer (Observer III) employs a switching mechanism between two observer configurations. Each configuration is almost equivalent to the explicit attitude observer in \cite{Mahony2008} except that a factor proportional to $\bar{\Theta}$ (in \eqref{R:vertheta:Theta:explicit}) is applied to the  input $\beta$ that is designed, in our case, based on the sum of vector-errors between the inertial measurements $b_i$ and their estimates \textit{perturbed} by the rotation matrix $\bar{\mathfrak{R}}$ (in \eqref{R:vertheta:Theta:explicit}). A key feature here is that, as the estimation error gets small, the values of $\bar{\mathfrak{R}}$ and $\bar{\Theta}$ approach the identity and the proposed hybrid scheme, {\it i.e.,} Observer III in Theorem~\ref{proposition::explicit}, becomes identical to the attitude observer proposed in \cite{Mahony2008}. On the other hand, for extremely large attitude estimation errors the perturbation matrix $\bar{\mathfrak{R}}$ becomes significant to guarantee the necessary gap between the two configurations. A similar remark can be made in regards of the attitude observer in \cite{Mahony2008} (with an obvious modification related to the vector measurements) and hybrid Observer I in Theorem~\ref{proposition::explicit}, which switches between six observer configurations.

\section{Simulation results}\label{section::simulation}
            In this section, we present numerical examples to validate our theoretical results. Consider system \eqref{kinematic} with  $$\omega(t)=\left[0.5\sin(0.1t),0.7\sin(0.2t+\pi),\sin(0.3t+\pi/3)\right]^{\top}$$ ($\mathrm{rad}/\sec$), and suppose that the measured angular velocity is given by \eqref{omega:measured} with the slowly varying bias
            $$b_\omega(t)=\big(1+0.1\cos(0.1t)\big)[0.003,-0.005,0.01]^{\top}.$$

            We also consider measurements of two non-collinear inertial vectors given by $a_1=[1,-1,1]^{\top}/\sqrt{3}$ and $a_2=[0,0,1]^{\top}$. We implement all the proposed hybrid observers \eqref{observer_hybrid}-\eqref{J}, with different choices of the potential function $\Phi(\tilde R,q)$, with $\hat{b}_{\omega}(t_0)=0$, $q(t_0)=1$, $\gamma_P=5$ and $\gamma_I=10$. The projection operator for the bias estimation law is implemented with a parameter bound $\bar b_\omega=0.1$. In all simulations, the selected initial attitude estimates $\hat{R}(t_0)$ lead to a large initial attitude estimation error $\tilde R(t_0)=\mathcal{R}_Q(0, e_1)$, with $e_1=(1,0,0)^\top$.
\subsection{Example 1}
            We implement the two hybrid observers (Observer I and Observer II corresponding to designs D1 and D2 and referred to as $\mathrm{HO_{I}}$ and $\mathrm{HO_{II}}$, respectively, in the figures below). The attitude matrix $R$ is reconstructed using any static attitude determination method such as SVD. The hysteresis gap of the hybrid switching mechanism is chosen as $\delta=0.8\Delta_{I}(k)$, for Observer I, and $\delta=0.8\Delta_{II}(k)$ for the Observer II, where the gain $k$ is selected in both cases as $k=0.95/\sqrt{5}$ such that condition \eqref{k_max} is satisfied.

            For comparison purposes, we also implement the following attitude observer (that we refer to as Observer $\mathrm{SO_{I}}$ standing for smooth observer I)
            \begin{align}\label{observer_smooth}
            \dot{\hat R}&=\hat R\left[\omega^y-\hat b_\omega+\gamma_P\beta\right]_\times,\quad \hat R(t_0)\in SO(3),\\
            \label{bias_smooth}
            \dot{\hat b}_\omega&=\mathrm{Proj}\left(-\gamma_I\beta,\hat b_\omega\right),\quad \hat b_\omega(t_0)\in\mathbb{R}^3,\\
            \label{beta_smooth}
            \beta&=\frac{1}{4}\hat R^\top\psi(\tilde R),
            \end{align}
            which is inspired by the attitude observer proposed in \cite{Mahony2008}. This smooth attitude observer can be obtained from the hybrid observer (Observer I) by setting $k=0$.

            The obtained results in this example are given in Figs.~\ref{Fig:attitude:error:Ex1}-\ref{Fig:bias:error:Ex1} showing, respectively, the attitude estimation error $|\tilde R|_I^2$ and the bias estimation error $\|\tilde b_\omega\|$. It can be seen from these figures that both hybrid attitude and gyro-bias observers ensure faster convergence of the estimation errors as compared to the traditional smooth observer \eqref{observer_smooth}-\eqref{beta_hybrid} despite the large initial attitude estimation error. Also, as mentioned in Remark~\ref{remark:diff:U or V}, the hybrid observer (Observer II) shows better performance (in terms of convergence) as compared to Observer I due to the nature of the potential function used in the design of each observer.
\begin{figure}[h!]
\centering
\includegraphics[scale=0.5]{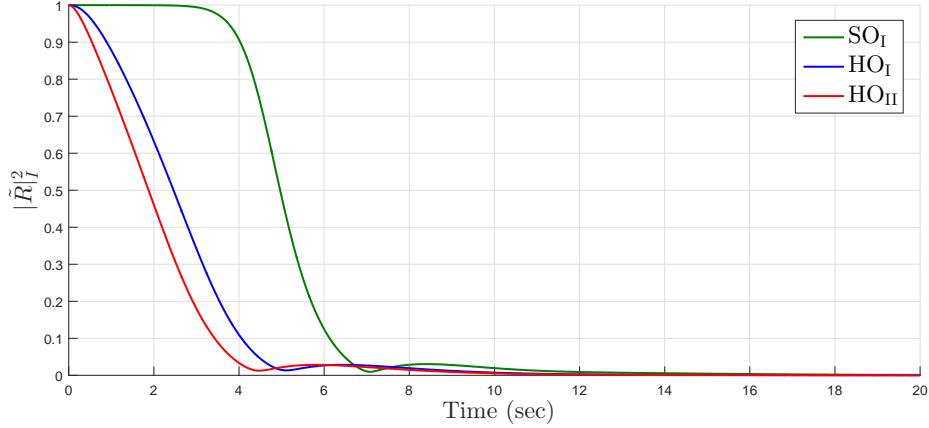}
\caption{\footnotesize{Example 1. Attitude estimation error.}}
\label{Fig:attitude:error:Ex1}
\end{figure}
\begin{figure}[h!]
\centering
\includegraphics[scale=0.5]{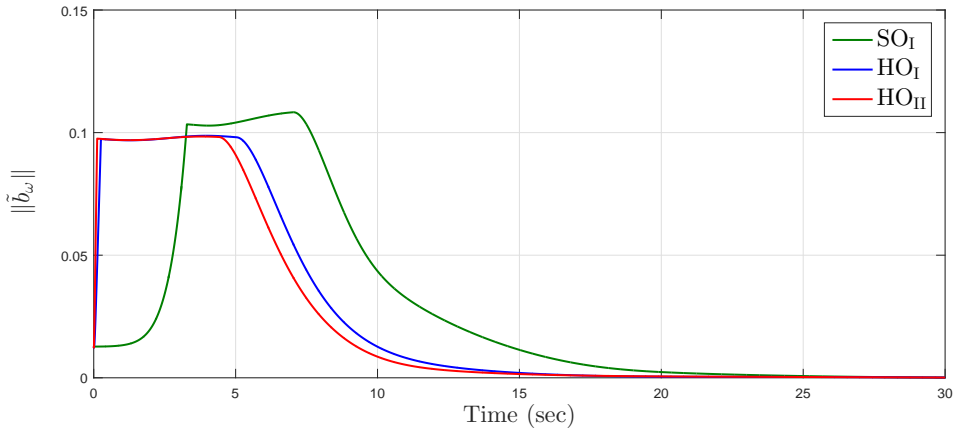}
\caption{\footnotesize{Example 1. Bias estimation error.}}
\label{Fig:bias:error:Ex1}
\end{figure}
\subsection{Example 2}
            In this second example, we implement the two hybrid observers (Observer III and Observer IV, referred to as $\mathrm{HO_{III}}$ and $\mathrm{HO_{IV}}$ in the figures below) given in Theorem~\ref{proposition::explicit}. Since only two vector measurements are assumed to be available, we construct a third vector $b_3=b_1\times b_2$ associated to the inertial vector $a_3=a_1\times a_2$ such that Assumption 1 is satisfied. Accordingly, we consider the matrix $A=\sum_{i=1}^{n}\rho_ia_ia_i^\top$ with $\rho_1=1, \rho_2=3$ and $\rho_3=1$. Note that the eigenvalues and eigenvectors of $A$ are used to determine the parameters of the hybrid observers in this example as presented in the design methods D3 and D4 in Lemma~\ref{lemma::synergism}. Similarly to the previous example, we select the hysteresis gap of the hybrid switching mechanism as $\delta=0.8\Delta_{III}(k)$, for Observer III, and $\delta=0.8\Delta_{IV}(k)$ for Observer IV, where the gain $k$ is selected as above so that condition \eqref{k_max} is verified.

            We also implement the following smooth attitude observer (that we refer to as $\mathrm{SO_{II}}$) for comparison purposes
\begin{align}\label{observer_smooth2}
            \dot{\hat R}&=\hat R\left[\omega^y-\hat b_\omega+\gamma_P\beta\right]_\times,\quad \hat R(t_0)\in SO(3),\\
            \label{bias_smooth2}
            \dot{\hat b}_\omega&=\mathrm{Proj}\left(-\gamma_I\beta,\hat b_\omega\right),\quad \hat b_\omega(t_0)\in\mathbb{R}^3,\\
            \label{beta_smooth2}
            \beta&=\frac{1}{8\lambda_{\max}^{\bar A}}\sum_{i=1}^3\rho_i\big(b_i\times\hat R^\top a_i\big),
\end{align}
            which is also inspired by \cite{Mahony2008} and \cite{grip2012observer}, and obtained from Observer III in Theorem~\ref{proposition::explicit} by setting $k=0$.

            The obtained results are given in Figs.~\ref{Fig:attitude:error:Ex2}-\ref{Fig:bias:error:Ex2}. Similarly to the previous example, one can deduce from these figures that Observer IV exhibits faster convergence as compared to Observer III, and both hybrid observers ensure better performance (in terms of convergence speed) as compared to the smooth estimation algorithm \eqref{observer_smooth2}-\eqref{beta_smooth2}.
\begin{figure}[h!]
\centering
\includegraphics[width = \columnwidth]{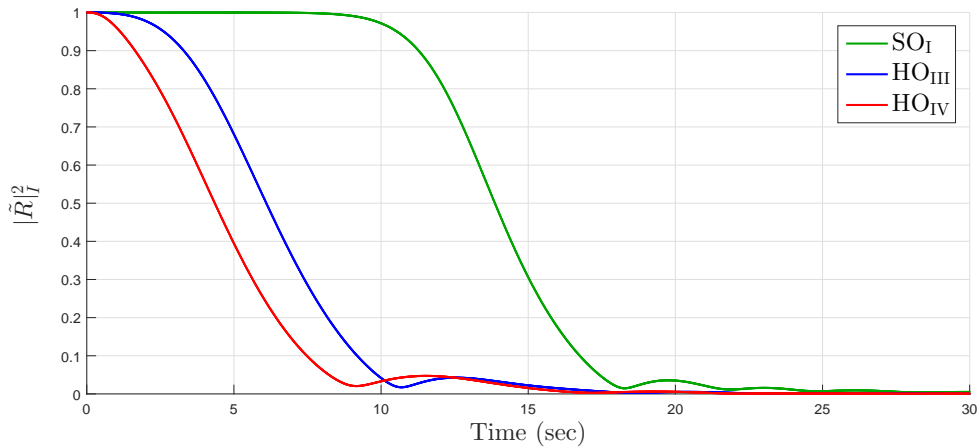}
\caption{\footnotesize{Example 2. Attitude estimation error.}}
\label{Fig:attitude:error:Ex2}
\end{figure}
\begin{figure}[h!]
\centering
\includegraphics[width = \columnwidth]{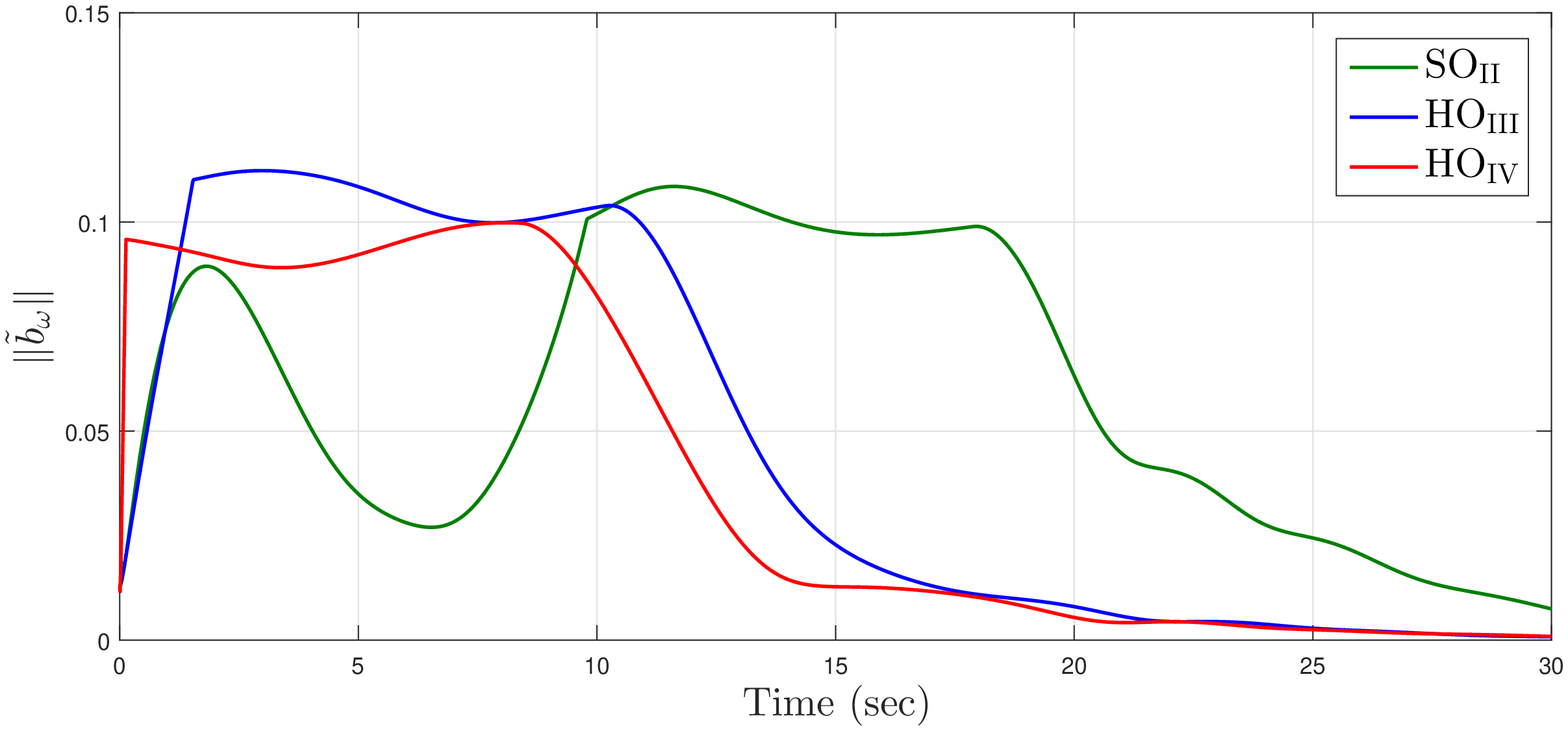}
\caption{\footnotesize{Example 2. Bias estimation error.}}
\label{Fig:bias:error:Ex2}
\end{figure}

\section{Conclusion}\label{conclusion}
        Nonlinear hybrid attitude and gyro-bias observers, leading to global exponential stability results, have been proposed. These observers rely on gyro and inertial vector measurements without the need for the reconstruction of the rotation matrix. Different sets of potential functions have been designed via an appropriate angular warping transformation applied to some smooth and non-smooth potential functions on $SO(3)$. These sets of potential functions are the backbones for the switching mechanisms involved in the four proposed hybrid observers. Numerical examples have been given to illustrate the performance of the proposed hybrid observers as compared to smooth (non-hybrid) observers,  inspired from the literature  in the case of large initial attitude estimation errors.

\appendix
\section{Proof of Lemma~\ref{lemma::derivatives}}\label{proof::lemma::derivatives}
Along the trajectories of $\dot X=X[u]_\times$, one has
\begin{align*}
\frac{d}{dt}\mathrm{tr}(A(I-X))=-\mathrm{tr}(A\dot X)&=-\mathrm{tr}(AX[u]_\times)\\
																				  &=\lls[u]_\times,AX\ggs\\
																				  &=\lls[u]_\times,\mathbb{P}_a(AX)\ggs\\
																				  &=2u^\top\psi(AX),
\end{align*}
where we have used identities \eqref{id::3}-\eqref{id::4}. Moreover, using \eqref{id::6} one has
\begin{align*}
\dot{\psi}(AX)=\frac{d}{dt}\left[\mathbb{P}_a(AX)\right]_\otimes&=\frac{1}{2}[AX[u]_\times+[u]_\times X^\top A]_\otimes,\\
 &=\frac{1}{2}[\mathrm{tr}(AX)I-X^\top A]u:=E(AX)u.
\end{align*}
\section{Proof of Lemma~\ref{lemma::identities_Q}}\label{proof::lemma::identities_Q}
Let $(\eta,\epsilon)\in\mathbb{Q}$ be the quaternion representation of the attitude matrix $X$. Using the Rodrigues formula (\ref{rod}) and \eqref{id::4} one obtains
\begin{align*}
\mathrm{tr}(A(I-R))&=\mathrm{tr}(A(-2[\epsilon]_\times^2-2\eta[\epsilon]_\times))\\
							 &=-2\mathrm{tr}(A[\epsilon]_\times^2)+2\eta\lls[\epsilon]_\times,A\ggs\\
							 &=-2\mathrm{tr}(A[\epsilon]_\times^2)+2\eta\lls[\epsilon]_\times,\mathbb{P}_a(A)\ggs\\
							 &=-2\mathrm{tr}(A[\epsilon]_\times^2)
\end{align*}
where we used $\mathbb{P}_a(A)=0$ since $A$ is symmetric. Now, using  \eqref{id::1}-\eqref{id::3} one gets
\begin{equation*}
\mathrm{tr}(A(I-R))=2\mathrm{tr}(\epsilon^{\top}\epsilon A-A\epsilon\epsilon^{\top})=2\epsilon^{\top}(\mathrm{tr}(A)I-A)\epsilon:=4\epsilon^\top\bar A\epsilon.
\end{equation*}
Again, using \eqref{rod} and \eqref{id::1}, one has
\begin{align*}
\mathbb{P}_a(AR)&=\frac{1}{2}(AR-R^\top A)\\
						   &=A\epsilon\epsilon^\top-\epsilon\epsilon^\top A+\eta A[\epsilon]_\times+\eta[\epsilon]_\times A\\
						   &=[\epsilon\times A\epsilon]_\times+2\eta[\bar{A}\epsilon]_\times,
\end{align*}
where \eqref{id::2} and \eqref{id::6} have been used. Consequently, one obtains
$$
\psi(AR)=\epsilon\times A\epsilon+\eta \bar{A}\epsilon=2(\eta I-[\epsilon]_\times)\bar{A}\epsilon.
$$
\section{Proof of Lemma~\ref{lemma::identities_SO(3)}}\label{proof::identities_SO(3)}
Let $(\eta,\epsilon)\in\mathbb{Q}$ be the quaternion representation of the attitude matrix $X$. In view of \eqref{id::normX} and \eqref{eq::Q::tr}, it is clear that $|X|_I^2=\|\epsilon\|^2$. Moreover, using again \eqref{eq::Q::tr}, one has
\begin{align*}
4\lambda_{\min}^{\bar A}|X|_I^2=4\lambda_{\min}^{\bar A}\|\epsilon\|^2\leq\mathrm{tr}(A(I-X))\leq4\lambda_{\max}^{\bar A}\|\epsilon\|^2=4\lambda_{\max}^{\bar A}|X|_I^2.
\end{align*}
Moreover, making use of \eqref{eq::Q::psi} and identity \eqref{id::1}, one obtains
\begin{equation*}
\begin{split}
\|\psi(AX)\|^2&=4\epsilon^{\top}\bar{A}(\eta I+[\epsilon]_\times)(\eta I-[\epsilon]_\times)\bar{A}\epsilon\\
									   &=4\epsilon^{\top}\bar{A}(\eta^2I-[\epsilon]_\times^2)\bar{A}\epsilon\\
									   &=4\epsilon^{\top}\bar{A}(I-\epsilon\epsilon^{\top})\bar{A}\epsilon\\
									   &=4\epsilon^{\top}\bar{A}^2\epsilon(1-\|\epsilon\|^2\cos^2(\phi)),
\end{split}
\end{equation*}
where the facts that $\eta^2+\epsilon^{\top}\epsilon=1$ and $\epsilon^\top\bar A\epsilon=\|\epsilon\|\|\bar A\epsilon\|\cos(\phi)$ with $\phi=\angle(\epsilon,\bar A\epsilon)$, have been used. Finally, in view of \eqref{eq::Q::tr}, one has
\begin{align*}
\mathrm{tr}(\underline{A}(I-X))=4\epsilon^{\top}\bar{A}^2\epsilon,
\end{align*}
where $\bar A^2=\frac{1}{2}(\mathrm{tr}(\underline{A})I-\underline{A})$ or, equivalently, $\underline{A}=\mathrm{tr}(\bar{A}^2)I-2\bar{A}^2$ which proves \eqref{eq::psi_norm}. Furthermore, since for any positive definite matrix $\bar A$ one has
\begin{align*}
\frac{\lambda_{\min}^{\bar A}}{\lambda_{\max}^{\bar A}}\leq\cos(u,\bar Au)\leq 1,
\end{align*}
which implies \eqref{ineq::alpha}. On the other hand, one has
\begin{align*}
           v^\top[\lambda_{\min}^{\bar A}I-E(AX)]v =&v[\lambda_{\min}^{\bar A}-\frac{1}{2}(\mathrm{tr}(AX)I-X^\top A)]v\\
							      =& (\lambda_{\min}^{\bar A}-\frac{1}{2}\mathrm{tr}(AX))\|v\|^2+\frac{1}{2}v^\top X^\top Av\\
							   \leq& (\lambda_{\min}^{\bar A}-\frac{1}{2}\mathrm{tr}(AX))\|v\|^2+\frac{1}{2}\lambda_{\max}^A\|v\|^2\\
                                =&\frac{1}{2}\mathrm{tr}(A(I-X))\|v\|^2,
            \end{align*}
            where the fact that $\lambda_{\min}^{\bar A}=\frac{1}{2}(\mathrm{tr}(A)-\lambda_{\max}^A)$ has been used. Finally,  using the fact that $\mathrm{tr}(AX)\leq\mathrm{tr}(A)$ for all $X\in SO(3)$, it can be verified that
\begin{align*}
            \|E(AX)\|^2_F =&~\mathrm{tr}\left(E(AX)^\top E(AX)\right)\\
            		    =&~\frac{1}{4}\mathrm{tr}\left(A^2+\mathrm{tr}^2(AX)I-\mathrm{tr}(AX)(AX+X^\top A)\right)\\
            			=&~\frac{1}{4}(\mathrm{tr}(A^2)+\mathrm{tr}^2(AX))\\
            			\leq&~\frac{1}{4}(\mathrm{tr}(A^2)+\mathrm{tr}^2(A))\\
            			=&~\frac{1}{4}\mathrm{tr}\left((\mathrm{tr}(A)I-A)(\mathrm{tr}(A)I-A)\right)\\
            			=&\|\bar A\|_F^2.
\end{align*}

\subsection{Proof of Lemma \ref{lemma::cross_prod}}\label{proof::lemma::cross_prod}
Making use of  identity \eqref{id::3} one has
\begin{equation*}
\begin{split}
\sum_{i=1}^{n}\rho_{i}\|X^\top v_i-Y^{\top}v_i\|^2&=\sum_{i=1}^{n}\rho_{i}v_i^{\top}(I-YX^{\top})(I-XY^\top)v_i\\
				   											 &=\sum_{i=1}^{n}\rho_{i}\mathrm{tr}(v_iv_i^{\top}(I-XY^\top))\\&=\mathrm{tr}(A(I-XY^\top)).
\end{split}
\end{equation*}
        Furthermore, making use of \eqref{id::2} one has
        \begin{align}\label{P1}
        \mathbb{P}_a(AXY^\top)&=\frac{1}{2}\sum_{i=1}^{n}\rho_{i}\left(v_iv_i^{\top}XY^\top-YX^\top v_iv_i^{\top}\right)\nonumber\\
        &=\frac{1}{2}\sum_{i=1}^{n}\rho_{i}Y\left(Y^\top v_iv_i^{\top}X-X^{\top}v_iv_i^{\top}Y\right)Y^\top\nonumber\\\nonumber
        &=\frac{1}{2}\sum_{i=1}^{n}\rho_{i}Y\left[(X^\top v_i\times Y^{\top}v_i)\right]_\times Y^\top,\\
        &=\frac{1}{2}\sum_{i=1}^{n}\rho_{i}\left[Y(X^\top v_i\times Y^{\top}v_i)\right]_\times,
        \end{align}
        where we used the property that $Y[v]_\times Y^\top=[Yv]_\times$ for all $Y\in SO(3)$ and $v\in\mathbb{R}^3$. Taking the $[\cdot]_\otimes$ operator on both sides of (\ref{P1}) yields \eqref{PA_bi}.

\subsection{Proof of Lemma~\ref{lemma::condition1}}\label{proof:lemma:conidition1}

Let $(\eta,\epsilon)$ and $(\eta_q,\epsilon_q)$ be, respectively, the unit quaternion representation of $\tilde R$ and $\Gamma_A(\tilde R,q)$. Using \eqref{Gamma}-\eqref{Rq} and \eqref{Qmultiply}-\eqref{R1R2}, one can deduce that
\begin{equation}
\epsilon_q=k\eta U_A(\tilde R)\nu(q)+\sqrt{1-k^2U_A^2(\tilde R)}\epsilon+kU_A(\tilde R)[\epsilon]_{\times}\nu(q).
\end{equation}
Taking the norm square of $\epsilon_q$, equality \eqref{Gamma} yields
\begin{align*}
\|\epsilon_q\|^2=&k^2\eta^2U_A^2(\tilde R)+(1-k^2U_A^2(\tilde R))\|\epsilon\|^2+\\
&k^2U_A^2(\tilde R)\|\epsilon\|^2\sin^2(\varphi_q)+\\
&2k\eta\|\epsilon\|U_A(\tilde R)\sqrt{1-k^2U_A^2(\tilde R)}\cos(\varphi_q),
\end{align*}
where $\varphi_q$ is the angle between $\epsilon$ and $\nu(q)$. In view of \eqref{UA} and \eqref{eq::Q::tr}, one has $U_A(\tilde R)\leq\|\epsilon\|^2$. Also since $\|\epsilon\|\in[0,1]$, one has $|\eta|\cdot\|\epsilon\|=\|\epsilon\|\sqrt{1-\|\epsilon\|^2}\leq1/2$. It follows that
\begin{align}\label{bound_eq_upper}
\|\epsilon_q\|^2&\leq\|\epsilon\|^2[1+k+k^2/4].
\end{align}
Moreover, in view of \eqref{k_max} and the fact that $U_A(\tilde R)\leq 1$, one has $kU_A(\tilde R)<1/\sqrt{2}$ and therefore
\begin{align}\nonumber
\|\epsilon_q\|^2&\geq\|\epsilon\|^2[1-k^2U^2_A(\tilde R)-kU_A(\tilde R)\sqrt{1-k^2U^2_A(\tilde R)}]\\
&\geq\|\epsilon\|^2[1-k^2-k\sqrt{1-k^2}],\label{bound_eq_lower}
\end{align}
where the fact that the scalar function $1-x^2-x\sqrt{1-x^2}$ is decreasing on the interval $x\in[0,1/\sqrt{2}]$ has been used to obtain the last inequality. Now, in view of \eqref{UA},\eqref{eq::Q::tr},\eqref{bound_eq_upper}, \eqref{bound_eq_lower} and the fact that $|\tilde R|_I^2=\mathrm{tr}(I-\tilde R)/4=\|\epsilon\|^2$, the result of Lemma \ref{lemma::condition1} follows.

\subsection{Proof of Lemma \ref{lemma::synergism}}\label{proof::lemma::synergism}
        We prove the result of Lemma \ref{lemma::synergism} for each design case.
        \subsubsection*{Case of D1}
        Let $\Phi=\Phi_{U_A}$ with $A=I$ and $\mathcal{Q}=\{1,\cdots,6\}$. Suppose that $\Gamma_A(\tilde R,q)\in S_\pi$ for $(\tilde R,q)\in SO(3)\times\mathcal{Q}$. Define $Q=(\eta,\epsilon)$, $Q_{q}=(\eta_{q},\epsilon_{q})$ and $Q_{p}=(\eta_{p},\epsilon_{p})$ as the unit quaternion representation of the rotation matrices $\tilde R$, $\Gamma_A(\tilde R,q)$ and $\Gamma_A(\tilde R,p)$, respectively, for some $p\in\mathcal{Q}$.

        Making use of the quaternion product rule \eqref{Qmultiply} and \eqref{R1R2}, and the definition of the map $\Gamma_A$ in \eqref{Gamma}-\eqref{Rq}, one has
        \begin{align}\label{etaq}
        \eta_q&=\eta\sqrt{1-k^2U_A(\tilde R)}-kU_A(\tilde R)\epsilon^{\top}\nu(q),\\
        \eta_p&=\eta\sqrt{1-k^2U_A(\tilde R)}-kU_A(\tilde R)\epsilon^{\top}\nu(p).\label{etap}
        \end{align}
        On the other hand, since $\Gamma_A(\tilde R,q)\in\mathcal{S}_\pi$ one has, in view of \eqref{eq::Q::tr}, $U_A(\Gamma_A(\tilde R,q))=\|\epsilon_q\|^2=1$ and hence $\eta_q=0$. Consequently, it follows from \eqref{etap} that
        \begin{align*}
        \eta_p&=kU_A(\tilde R)\epsilon^{\top}\nu(q)-kU_A(\tilde R)\epsilon^{\top}\nu(p)\\
        		  &=kU_A(\tilde R)\epsilon^\top(\nu(q)-\nu(p)).
        \end{align*}
        Also, it can be verified, using \eqref{eq::Q::tr}, that $U_A(\tilde R)=\|\epsilon\|^2$, which yields
        \begin{equation}\label{eta_p}
        \eta_p=k\|\epsilon\|^3\left(\cos(\vartheta(q))-\cos(\vartheta(p))\right),
        \end{equation}
        where the fact that $\nu(q)^{\top}\epsilon=\|\epsilon\|\cos(\vartheta(q))$, such that $\vartheta(q)=\angle (\nu(q),\epsilon)$, has been used. On the other hand, in view of \eqref{etaq} and the fact that $\eta_q=0$ and $U_A(\tilde R)=\|\epsilon\|^2$, one obtains
        \begin{equation*}
        \sqrt{1-\|\epsilon\|^2}\sqrt{1-k^2\|\epsilon\|^4}=\|\epsilon\|^3|\cos(\vartheta(q))|,
        \end{equation*}
        where $|\eta|=\sqrt{1-\|\epsilon\|^2}$ has been used.
Squaring both sides of the above equation, it follows that
        $1-\|\epsilon\|^2=k^2\|\epsilon\|^4(1-\sin^2(\vartheta(q))\|\epsilon\|^2)$ which results in the quadratic inequality
$$
1-\|\epsilon\|^2-k^2\|\epsilon\|^2\leq 0,
$$
where $\|\epsilon\|^2\in [0,1]$. This inequality has a solution satisfying
        \begin{align}\label{e_lower_bound}
        \|\epsilon\|^2\geq\frac{-1+\sqrt{1+4k^2}}{2k^2}.
        \end{align}
        Besides, since $\nu(q+3)=-\nu(q+3)$, it follows that $\vartheta(q+3)=\pi+\vartheta(q)$ for all $q\in\mathcal{Q}$. Consequently, using relation
        $$
        \max_{y}\;\{|x+y|,|x-y|\}=|x|+\max_{y}|y|,
        $$
        one can show that
            {\small{\begin{align*}
            &\max_{p\in\mathcal{Q}}\big|\cos(\vartheta(q))-\cos(\vartheta(p))\big|\\
            &=\max_{p\in\{1,2,3\}}\Big\{\big|\cos(\vartheta(q))
                         -\cos(\vartheta(p))\big|,\big|\cos(\vartheta(q))+|\cos(\vartheta(p))\big|\Big\}\\
            &=\big|\cos(\vartheta(q))\big|+\max_{p\in\{1,2,3\}}\big|\cos(\vartheta(p))\big|\geq\frac{1}{\sqrt{3}}
            \end{align*}}}
\noindent where the fact that
        $$
            \sum_{p=1}^3\cos^2(\vartheta(p))=1 \; \leq 3\;\max_{p\in\mathcal{Q}}|\cos(\vartheta(p))|^2,
        $$
        has been used due to the orthogonality of $\{\nu(1),\nu(2),\nu(3)\}$. Consequently, in view of \eqref{eta_p} and \eqref{e_lower_bound} and the above result, one obtains
$$
\underset{p\in\mathcal{Q}}{\mathrm{max}}\;|\eta_p|^2\geq\frac{[-1+\sqrt{1+4k^2}]^{3}}{24k^4},
$$
Therefore, it can be shown that, for all $\Gamma_A(\tilde R,q)\in S_\pi$, the following holds
       \begin{align*}
       U_A(\Gamma_A(\tilde R,q))-\min_{p\in\mathcal{Q}}U_A(\Gamma_A(\tilde R,p))&=\|\epsilon_q\|^2-\min_{p\in\mathcal{Q}}\|\epsilon_p\|^2\\
      &=1-\min_{p\in\mathcal{Q}}\big(1-\eta_p^2\big)\\
      &=\underset{p\in\mathcal{Q}}{\mathrm{max}}\;\eta_p^2\geq\Delta_{I}(k)>\delta.
        \end{align*}
        As a result, one can conclude that if $\Gamma_A(\tilde R,q)\in S_\pi$ then $(\tilde R,q)\notin\mathcal{F}$. By contraposition, for all $(\tilde R,q)\in\mathcal{F}$, one has $\Gamma_A(\tilde R,q)\notin S_\pi$.

\subsubsection*{Case of D2}
       Let $\Phi=\Phi_{V_A}$ with $A=I$ and $\mathcal{Q}=\{1,\cdots,6\}$. Since $A=I$, the set of all eigenvectors of $A$ is identified by $\mathcal{E}(A)=\mathbb{S}^2$. Moreover, following similar steps as in the proof of D1, and for all $\Gamma_A(\tilde R,q)\in S_\pi$, one has
        \begin{align*}
        &V_A(\Gamma_A(\tilde R,q))-\min_{p\in\mathcal{Q}}V_A(\Gamma_A(\tilde R,p))\\
        &=2-\min_{p\in\mathcal{Q}}2[1-\sqrt{1-\|\epsilon_p\|^2}]\\
        &=2\max_{p\in\mathcal{Q}}|\eta|_p\geq 2\sqrt{\Delta_{I}(k)}=\Delta_{II}(k)>\delta.
        \end{align*}
It follows that if $\Gamma_A(\tilde R,q)\in S_\pi$ then $(\tilde R,q)\notin\mathcal{F}$. It follows, by contraposition, that for all $(\tilde R,q)\in\mathcal{F}$, one has $\Gamma_A(\tilde R,q)\notin S_\pi$.
%
     \subsubsection*{Case of D3}
In \cite{berkane2015construction}, we have shown that the function $\Phi_{U_A}$ with the parameters selected as in D3 satisfies:
 \begin{equation}\label{pf::case3}
       U_A(\Gamma_A(\tilde R,q))-\min_{p\in\mathcal{Q}}U_A(\Gamma_A(\tilde R,p))\geq\Delta_{III}(A,k),
        \end{equation}
        for all $(\tilde R,q)$ satisfying $\Gamma_A(\tilde R,q)=\mathcal{R}_Q(0,v)$ where $v$ is an eigenvector of $A$. Since $0<\delta<\Delta_{III}(A,k)$, it is clear that the set where $\Gamma_A(\tilde R,q)=\mathcal{R}_Q(0,v)$ lies entirely in the jump set $\mathcal{J}$. Hence, the attitude $\Gamma_A(\tilde R,q)$ can not be equal to $\mathcal{R}_Q(0,v)$ for any $v\in\mathbb{S}^2$ (eigenvector of $A$) during the flows of $\mathcal{F}$.
     \subsubsection*{Case of D4}
           Let $(\tilde R,q)$ satisfying $\Gamma_A(\tilde R,q)=\mathcal{R}_Q(0,v)$ where $v$ is an eigenvector of $A$, or equivalently $\Gamma_A(\tilde R,q)\in\mathcal{S}_\pi$. In view of \eqref{pf::case3} and using the fact that $\xi\leq U_A(\mathcal{R}_Q(0,v))\leq 1$, one can conclude that
         \begin{align*}
         &U_A(\Gamma_A(\tilde R,q))-\min_{p\in\mathcal{Q}}U_A(\Gamma_A(\tilde R,p))\\
         &\geq\Delta_{III}(A,k)\\
         &=\Delta_{IV}(A,k)^2/4+\Delta_{IV}(A,k)\sqrt{1-\xi}
         \\
         &\geq\Delta_{IV}(A,k)^2/4+\Delta_{IV}(A,k)\sqrt{1-U_A(\Gamma_A(\tilde R,q))},
\end{align*}
 Hence, by completing the squares, one obtains
\begin{align*}
                     &\max_{p\in\mathcal{Q}}\sqrt{1-U_A(\Gamma_A(\tilde R,p))}-\sqrt{1-U_A(\Gamma_A(\tilde R,q))}\\
                     &\geq\Delta_{IV}(A,k)/2,
 \end{align*}
  or equivalently
\begin{equation*}
       V_A(\Gamma_A(\tilde R,q))-\min_{p\in\mathcal{Q}}V_A(\Gamma_A(\tilde R,p))\geq\Delta_{IV}(A,k).
\end{equation*}
Hence, if $\delta<\Delta_{IV}(A,k)$ then it is obvious that the set where $\Gamma_A(\tilde R,q)=\mathcal{R}_Q(0,v)$ lies entirely in the jump set $\mathcal{J}$.\\

Now, let us show that $\mathcal{F}\subseteq\mathcal{D}$ for all the cases. For D1 and D3, the potential function $\Phi=\Phi_{U_A}$ is differentiable on all $SO(3)\times\mathcal{Q}$ due to the fact that $U_A$ is smooth on $SO(3)$ and the transformation $\Gamma_A$ is differentiable everywhere as shown in Lemma \ref{lemma::condition1}. Thus, $\mathcal{F}\subseteq\mathcal{D}=SO(3)\times\mathcal{Q}$ holds. The potential function $\Phi_{V_A}$, however, is differentiable on the set
\begin{align*}
\mathcal{D}=\{(\tilde R,q)\in SO(3)\times\mathcal{Q}\mid\Phi_{U_A}(\tilde R,q)\neq1\}.
\end{align*}
 Let $(\tilde R,q)\in SO(3)\times\mathcal{Q}$ such that $\Gamma_A(\tilde R,q)\in S_\pi$. Hence, one has
$$
\Phi_{U_A}(\tilde R,q)=U_A(\Gamma(\tilde R,q))=U_A(\mathcal{R}_Q(0,v))=1,
$$
for some $v\in\mathbb{S}^2$. Consequently, in this case $(\tilde R,q)\notin\mathcal{D}$. Therefore, since $\Gamma_A(\tilde R,q)\notin\mathcal{S}_\pi$ is guaranteed during the flows of $\mathcal{F}$ one has $\mathcal{F}\subseteq\mathcal{D}$.

\bibliographystyle{IEEETran}
\bibliography{IEEEabrv,Hybrid}

\begin{thebibliography}{10}
\providecommand{\url}[1]{#1}
\csname url@samestyle\endcsname
\providecommand{\newblock}{\relax}
\providecommand{\bibinfo}[2]{#2}
\providecommand{\BIBentrySTDinterwordspacing}{\spaceskip=0pt\relax}
\providecommand{\BIBentryALTinterwordstretchfactor}{4}
\providecommand{\BIBentryALTinterwordspacing}{\spaceskip=\fontdimen2\font plus
\BIBentryALTinterwordstretchfactor\fontdimen3\font minus
  \fontdimen4\font\relax}
\providecommand{\BIBforeignlanguage}[2]{{%
\expandafter\ifx\csname l@#1\endcsname\relax
\typeout{** WARNING: IEEEtran.bst: No hyphenation pattern has been}%
\typeout{** loaded for the language `#1'. Using the pattern for}%
\typeout{** the default language instead.}%
\else
\language=\csname l@#1\endcsname
\fi
#2}}
\providecommand{\BIBdecl}{\relax}
\BIBdecl

\bibitem{Shuster1981}
M.~D. Shuster and S.~D. Oh, ``Three-axis attitude determination from vector
  observations,'' \emph{Journal of Guidance and Control}, vol.~4, pp. 70--77,
  1981.

\bibitem{Markley1988}
F.~Markley, ``Attitude determination using vector observations and the singular
  value decomposition,'' \emph{Journal of the Astronautical Sciences}, vol.~36,
  pp. 245--258, 1988.

\bibitem{Markley2003}
------, ``Attitude error representations for kalman filtering,'' \emph{Journal
  of Guidance, Control, and Dynamics}, vol.~63, no.~2, pp. 311--317, 2003.

\bibitem{crassidis2007survey}
J.~Crassidis, F.~Markley, and Y.~Cheng, ``{Survey of nonlinear attitude
  estimation methods},'' \emph{Journal of Guidance Control and Dynamics},
  vol.~30, no.~1, p.~12, 2007.

\bibitem{Mahony2008}
R.~Mahony, T.~Hamel, and J.-M. Pflimlin, ``Nonlinear complementary filters on
  the special orthogonal group,'' \emph{IEEE Transactions on Automatic
  Control}, vol.~53, no.~5, pp. 1203--1218, June 2008.

\bibitem{tayebi2006attitude}
A.~Tayebi and S.~McGilvray, ``Attitude stabilization of a vtol quadrotor
  aircraft,'' \emph{IEEE Transactions on Control Systems Technology}, vol.~14,
  no.~3, pp. 562--571, 2006.

\bibitem{Bhat2000}
D.~S.~B. Sanjay P.~Bhat, ``A topological obstruction to continuous global
  stabilization of rotational motion and the unwinding phenomenon,''
  \emph{Systems \& Control Letters}, vol.~39, pp. 63--70, 2000.

\bibitem{lee2012}
T.~Lee, ``Exponential stability of an attitude tracking control system on
  {$SO(3)$} for large-angle rotational maneuvers,'' \emph{Systems \& Control
  Letters}, vol.~61, no.~1, pp. 231--237, 2012.

\bibitem{zlotnik2017nonlinear}
D.~E. Zlotnik and J.~R. Forbes, ``Nonlinear estimator design on the special
  orthogonal group using vector measurements directly,'' \emph{IEEE
  Transactions on Automatic Control}, vol.~62, pp. 149--160.

\bibitem{mayhew2011hybrid}
C.~G. Mayhew and A.~R. Teel, ``Hybrid control of rigid-body attitude with
  synergistic potential functions,'' in \emph{American Control Conference},
  2011, pp. 287--292.

\bibitem{lee2015tracking}
T.~Lee, ``Global exponential attitude tracking controls on {$SO(3)$},''
  \emph{IEEE Transactions on Automatic Control}, vol.~60, no.~10, pp.
  2837--2842, 2015.

\bibitem{berkaneCDC2015synergistic}
S.~Berkane and A.~Tayebi, ``On the design of synergistic potential functions on
  {$SO(3)$},'' in \emph{the 54th IEEE Conference on Decision and Control,
  Osaka, Japan}, 2015, pp. 270--275.

\bibitem{berkane2015construction}
------, ``Construction of synergistic potential functions on {$SO(3)$} with
  application to velocity-free hybrid attitude stabilization,'' \emph{IEEE
  Transactions on Automatic Control}, vol.~62, no.~1, pp. 495--501, 2017.

\bibitem{lee2015observer}
E.~K. Tse-Huai~Wu and T.~Lee, ``Globally asymptotically stable attitude
  observer on {$SO(3)$},'' in \emph{the 54th IEEE Conference on Decision and
  Control, Osaka, Japan}, 2015, pp. 2164--2168.

\bibitem{batista2012sensor}
P.~Batista, C.~Silvestre, and P.~Oliveira, ``Sensor-based globally
  asymptotically stable filters for attitude estimation: Analysis, design, and
  performance evaluation,'' \emph{IEEE Transactions on Automatic Control},
  vol.~57, no.~8, pp. 2095--2100, 2012.

\bibitem{Batista2012}
------, ``Globally exponentially stable cascade observers for attitude
  estimation,'' \emph{Control Engineering Practice}, vol.~20, no.~2, pp.
  148--155, 2012.

\bibitem{berkaneACC2016observer}
S.~Berkane, A.~Abdessameud, and A.~Tayebi, ``Global hybrid attitude estimation
  on the special orthogonal group {$SO(3)$},'' in \emph{The 2016 American
  Control Conference, Boston, MA, USA}, pp. 113--118.

\bibitem{berkaneCDC2016observer}
------, ``A globally exponentially stable hybrid attitude and gyro-bias
  observer,'' in \emph{the 55th IEEE Conference on Decision and Control, Las
  Vegas, USA}, 2016, pp. 308--313.

\bibitem{Shuster1993}
M.~Shuster, ``A survey of attitude representations,'' \emph{The Journal of the
  Astronautical Sciences}, vol.~41, no.~4, pp. 439--517, 1993.

\bibitem{Goebel2006}
R.~Goebel and A.~Teel, ``Solutions to hybrid inclusions via set and graphical
  convergence with stability theory applications,'' \emph{Automatica}, vol.~42,
  pp. 573--587, 2006.

\bibitem{Goebel2009}
R.~Goebel, R.~G. Sanfelice, and A.~R. Teel, ``Hybrid dynamical systems,''
  \emph{IEEE Control Systems Magazine}, vol.~29, no.~2, pp. 28--93, 2009.

\bibitem{Teel2007}
R.~G. Sanfelice, R.~Goebel, and A.~Teel, ``Invariance principles for hybrid
  systems with connections to detectability and asymptotic stability,''
  \emph{IEEE Transactions on Automatic Control}, vol.~52, no.~12, pp.
  2282--2297, 2007.

\bibitem{teel2013lyapunov}
A.~Teel, F.~Forni, and L.~Zaccarian, ``Lyapunov-based sufficient conditions for
  exponential stability in hybrid systems,'' \emph{IEEE Transactions on
  Automatic Control}, 2013.

\bibitem{marino1998robust}
R.~Marino and P.~Tomei, ``Robust adaptive state-feedback tracking for nonlinear
  systems,'' \emph{IEEE Transactions on Automatic Control}, vol.~43, no.~1, pp.
  84--89, 1998.

\bibitem{Koditschek}
D.~E. Koditschek, ``Application of a new lyapunov function to global adaptive
  attitude tracking,'' in \emph{The 27th Conference on Decision and Control,
  Austin, Texas}, 1988.

\bibitem{Sanyal2009}
A.~Sanyal, A.~Fosbury, N.~Chaturvedi, and D.~Bernstein, ``Inertia-free
  spacecraft attitude tracking with disturbance rejection and almost global
  stabilization,'' \emph{Journal of Guidance, Control, Dynamics}, vol.~32,
  no.~4, pp. 1167--1178, 2009.

\bibitem{Saccon2010}
A.~Saccon, J.~Hauser, and A.~Aguiar, ``Exploration of kinematic optimal control
  on the lie group so(3),'' in \emph{8th IFAC Symposium on Nonlinear Control
  Systems}, 2010.

\bibitem{grip2012observer}
H.~F. Grip, T.~I. Fossen, T.~A. Johansen, and A.~Saberi, ``Attitude estimation
  using biased gyro and vector measurements with time-varying reference
  vectors,'' \emph{IEEE Transactions on Automatic Control}, vol.~57, no.~5, pp.
  1332--1338, May 2012.

\end{thebibliography}
\end{document}